\documentclass[10pt]{article}

\usepackage{amsfonts}
\usepackage{amsmath,amsthm,amscd,amssymb,mathrsfs,setspace}
\usepackage{latexsym,epsf,epsfig}
\usepackage{color}
\usepackage[hmargin=3cm,vmargin=3.5cm]{geometry}
\usepackage{graphicx}
\newcommand{\ds}{\displaystyle}

\newcommand{\xb}{{\bf{x}}}

\newcommand{\Dn}{\partial_{\nu}}

\newcommand{\cF}{\mathcal{F}(z)}
\newcommand{\cG}{\mathcal{G}(z)}

\theoremstyle{plain}
\newtheorem{theorem}{Theorem}[section]
\newtheorem{lemma}[theorem]{Lemma}
\newtheorem{proposition}[theorem]{Proposition}

\newtheorem{assumption}{Assumption}
\theoremstyle{remark}
\newtheorem{remark}{Remark}[section]
\numberwithin{equation}{section}
\numberwithin{theorem}{section}
\numberwithin{remark}{section}
\begin{document}

\title{Smooth attractors of finite dimension for von Karman evolutions with
nonlinear frictional damping\\
localized in a boundary layer.}
 \author{\begin{tabular}[t]{c@{\extracolsep{2em}}c@{\extracolsep{2em}}c}
           Pelin G. Geredeli  & Irena Lasiecka & Justin T. Webster \\
\it       Hacettepe University & \it University of Virginia & \it University of Virginia \\
\it        Ankara, Turkey & \it Charlottesville, VA &\it Charlotesville, VA\\
\it       pguven@hacettepe.edu.tr & \it il2v@virginia.edu & \it jtw3k@virginia.edu
\end{tabular}}
\maketitle

\begin{abstract}
\noindent In this paper dynamic von Karman equations with localized interior damping
supported in a boundary collar are considered. Hadamard well-posedness for
von Karman plates with various types of nonlinear damping are well-known,
and the long-time behavior of nonlinear plates has been a topic of recent
interest. Since the von Karman plate system is of "hyperbolic type" with
\textit{critical nonlinearity} (noncompact with respect to the phase space),
this latter topic is particularly challenging in the case of \textit{%
geometrically constrained} and \textit{nonlinear} damping. In this paper we
first show the existence of a compact global attractor for finite-energy
solutions, and we then prove that the attractor is both \textit{smooth} and
finite dimensional. Thus, the hyperbolic-like flow is stabilized
asymptotically to a smooth and finite dimensional set.
\vskip.3cm
\noindent Key terms: dynamical systems, long-time behavior, global attractors,
nonlinear plates, nonlinear damping, localized damping
\end{abstract}

\section{Introduction}

We consider the evolution of a nonlinear von Karman plate subject to
nonlinear frictional damping with essential support in a boundary collar.
Our aim is to consider the long-time behavior of the corresponding
evolution. This includes studying (a) existence of global attractor which
captures long-time behavior of the dynamics, and (b) properties of this
attractor, such as smoothness and finite dimensionality.

In short, our goal is to show that the original infinite dimensional and
non-smooth dynamics of hyperbolic type can be reduced (asymptotically) to a
finite dimensional and regular set, with respect to the topology of ``finite
energy". The latter is associated with weak (or generalized) solutions
of the underlying semigroup for the dynamics. This type of result then allows the
implementation of tools from finite dimensional control theory in order to
achieve a preassigned outcome for the dynamics.

\subsection{Model and Energies}

Let $\Omega \subset \mathbb{R}^2$ be a bounded domain with $\partial \Omega
= \Gamma$ taken to be sufficiently smooth. We consider a plate model where
the real-valued function $u(x,y;t)$ models the out-of-plane displacement of
a plate with negligible thickness. Then the von Karman model \cite
{ciarlet,lagnese} requires that $u$ satisfies
\begin{align}  \label{plate}
u_{tt}+\Delta^2 u +d(\mathbf{x})g(u_t)= f_V(u)+p & ~~~\text{ in }~~ \Omega
\times (0,\infty) \equiv Q \\
u\big|_{t=0} = u_0,~~ u_t\big|_{t=0}=u_1.  \notag
\end{align}
The von Karman nonlinearity
\begin{equation}  \label{vonK}
f_V=[v(u)+F_0,u]
\end{equation}
is given in terms of (a) the Airy Stress function $v(u)$, satisfying
\begin{align}  \label{airy}
\Delta^2v(u)=-[u,u] & ~\text{in}~~\Omega \\
\partial_{\nu} v(u)=v(u) = 0 & ~\text{on}~~\Gamma,  \notag
\end{align}
and (b) the von Karman bracket given by
\begin{equation}  \label{bracket}
[u,w]=u_{xx}w_{yy}+u_{yy}w_{xx}-2u_{xy}w_{xy}.
\end{equation}

The internal force $F_0 \in H^{\theta}(\Omega)\cap H_0^1(\Omega) $, $\theta
> 3 $, and external force $p \in L_2(\Omega)$ play an essential role in
shaping the nontrivial stationary solutions. (In this
paper $H^s(D)$ denotes the Sobolev space of order $s \in \mathbb{R}$ on
domain $D$.) In the absence of these
forces, the stationary solution of the corresponding nonlinear boundary
value problem becomes trivial and simply reduces to zero.

In this treatment we focus on the stabilizing properties of the damping term
$d(\mathbf{x})g(u_{t})$. In particular, we take $g(\cdot )\in C(\mathbb{R})$
to be a monotone increasing function, with $g(0)=0$ and further boundedness and
smoothness assumptions to be imposed later; additionally, $d(\mathbf{x})\equiv
d_{\omega }(\mathbf{x})$ is a nonnegative $L_{\infty }(\Omega )$
localizing function which restricts the damping term $g(u_{t})$ to a
particular subset $\omega \subset \Omega $. This is to say $\omega \subset
\text{supp}~d$ or $d(x)\geq c_{0}>0$ for $x\in \omega $. Initially we will take $%
\omega $ to be a general set $\omega \subset \subset \Omega $, but more
specifically, we are interested in taking $\omega $ to be an open collar of
the boundary $\Gamma $. This type of damping represents localized, viscous
damping \textit{near} the boundary $\Gamma $.

The boundary conditions we consider for the plate are:

\begin{enumerate}
\item Clamped, denoted \textbf{(C)}
\begin{equation}  \label{clamp}
u=\partial_{\nu} u = 0 ~~\text{ in }~~ \Gamma \times (0,\infty) \equiv
\Sigma.
\end{equation}

\item Hinged (simply-supported), which we denote by \textbf{(H)}
\begin{equation}  \label{hinge}
u=\Delta u = 0 ~~ \text{ in } ~~ \Sigma.
\end{equation}

\item Free-type, denoted by \textbf{(F)}
\begin{eqnarray}  \label{free}
\mathcal{B}_1 u \equiv \Delta u + (1-\mu) B_1 u = 0~~\mathrm{on}~~ \Gamma_1,
\notag \\
\mathcal{B}_2 u \equiv \partial_{\nu} \Delta u + (1-\mu) B_2 u - \mu_1 u -
\beta u^3 = 0 ~~\mathrm{on}~~ \Gamma_1,  \notag \\
u = \partial_{\nu} u =0 ~\text{(clamped)}~~\mathrm{on}~~ \Gamma_0,
\end{eqnarray}
where we have partitioned the boundary $\Gamma = \Gamma_0 \cup \Gamma_1$
(with $\Gamma_0 $ possibly empty). For simplicity we assume that $\overline{
\Gamma}_0\cap\overline{\Gamma}_1 =\emptyset$. Otherwise the regularity theory for
elliptic problems with \textit{mixed} boundary conditions must be
invoked. The boundary operators $B_1$ and $B_2$ are given by \cite{lagnese}:
\begin{equation*}
\begin{array}{c}
B_1u = 2 \nu_1\nu_2 u_{xy} - \nu_1^2 u_{yy} - \nu_2^2 u_{xx}\;, \\
\\
B_2u = \partial_{\tau} \left[ \left( \nu_1^2 -\nu_2^2\right) u_{xy} + \nu_1
\nu_2 \left( u_{yy} - u_{xx}\right)\right]\,,%
\end{array}%
\end{equation*}
where $\nu=(\nu_1, \nu_2)$ is the outer normal to $\Gamma$, $\tau= (-\nu_2,
\nu_1)$ is the unit tangent vector along $\Gamma$. The parameters $\mu_1$
and $\beta $ are nonnegative, the constant $0<\mu<1$ has the meaning of the
Poisson modulus.
\end{enumerate}

\noindent\textbf{Notation:} Note, when referencing the plate equation above
in (\ref{plate}) we will write (\ref{plate})(C), (\ref{plate})(H), or (\ref%
{plate})(F) to indicate which boundary conditions we are taking. We write the norm in $H^s(D)$  as $||\cdot||_s$ and $||\cdot||_0 \equiv
||\cdot||_{L_2(D)} $; for simplicity (when the meaning is clear from
context) norms and inner products written without subscript (~$(\cdot,\cdot)
$, $||\cdot||$~), are taken to be $L_2(D)$ of the appropriate domain $D$.
Additionally, we employ the notation that $H^s_0(D)$ gives the closure of $%
C_0^{\infty}(D)$ in the $||\cdot||_s$ norm.

\medskip

The von Karman plate equation is well-known in nonlinear elasticity, and
constitutes a basic model to describe the nonlinear oscillations of a thin
plate with large displacements \cite{lagnese} (and references therein). In
particular, we take the thickness of the plate to be negligible (as is usual
in the modeling of thin structures \cite{ciarlet}).

\begin{remark}
It is worth noting that the von Karman plate model can accomodate plates
with non-negligible thickness - the equation in (\ref{plate}) then gives the
vertical displacement of the central plane of the plate. This is tantamount
to adding the term $-\gamma\Delta u_{tt}$, $\gamma>0$ to the LHS of (\ref%
{plate}). This term corresponds to rotational inertia in the filaments of
the plate, and (a) is \textit{regularizing} from the energetic point of view
and (b) forces the dynamics of the plate to be hyperbolic. In this treatment
we take $\gamma =0$, since it constitutes the most challenging problem
mathematically, however, a future manuscript will address the case $\gamma >
0$ and the limiting problem (convergence of solutions and attractors) as $%
\gamma \searrow 0$.
\end{remark}

The energies associated to the above equation are given by (in the case of
clamped (C) or hinged (H) boundary conditions)
\begin{align*}  \label{energy}
E(t)=&\dfrac{1}{2}\big(||\Delta u||^{2}+||u_{t}||^{2}\big), \\
\widehat{E}(t)=&E(t)+\dfrac{1}{4}||\Delta v(u)||^{2}  \notag \\
\mathscr{E} (t)=&E(t)+\Pi (u),  \notag
\end{align*}
where
\begin{equation}
\Pi (u)=\dfrac{1}{4}\int_{\Omega }\big(|\Delta v(u)|^{2}-2[F_{0},u]u-4pu %
\big).  \label{pi}
\end{equation}

The above (linear) energy $E(t)$ dictates our state space $\mathcal{H}$,
which depends on boundary conditions. In the case of clamped boundary
conditions (C) we have $\mathcal{H}_{1}\equiv H_{0}^{2}(\Omega )\times
L_{2}(\Omega )$. For hinged boundary conditions (H) we have $\mathcal{H}
_{2}\equiv (H^{2}\cap H_{0}^{1})(\Omega )\times L_{2}(\Omega )$.

Lastly, for free boundary conditions (F) we have $\mathcal{H}_{3}\equiv
(H^{2}\cap H_{0,\Gamma _{0}}^{2})(\Omega )\times L_{2}(\Omega )$ (where $%
H_{0,\Gamma _{0}}^{2}(\Omega )$ is the Sobolev space $H^{2}(\Omega )$ with
clamped conditions on $\Gamma _{0}$); the potential energy in this case is
given by the bilinear form
\begin{equation}
a(u,v)=\int_{\Omega }\widetilde{a}(u,v)+\mu _{1}\int_{\Gamma _{1}}uv,
\label{a-uw}
\end{equation}%
where
\begin{equation}
\widetilde{a}(u,v)\equiv u_{xx}v_{xx}+u_{yy}v_{yy}+\mu
(u_{xx}v_{yy}+u_{yy}v_{xx})+2(1-\mu )u_{xy}v_{xy}.  \label{a-tild}
\end{equation}%
Then the energy becomes
\begin{equation*}
E(t)=\frac{1}{2}\big\{ ||u_{t}||^{2}+a\big(u(t),u(t)\big)\big\}
\end{equation*}%
\begin{equation*}
\widehat{E}(t)\equiv E(t)+\frac{1}{4}||\Delta v(u)||^{2}+\frac{\beta }{2}%
\int_{\Gamma _{1}}u^{4}d\Gamma .
\end{equation*}%
The total energy becomes
\begin{equation*}
\mathscr{E}(t)=E(t)+\Pi (u(t))+\frac{1}{4}\beta \int_{\Gamma _{1}}u^{4}(t).
\end{equation*}

\begin{remark}
We note that this last form of the energy described by the bilinear form $
a(u,v) $ can also be applied to clamped or hinged boundary conditions.
Indeed, in this latter case the bilinear form $a(u,u) $ collapses just to $
||\Delta u||^2 $.
\end{remark}

It will be convenient to introduce an elliptic operator $A$$:\mathscr{D}
(A)\subset L_{2}(\Omega )\to L_{2}(\Omega )$ given by $Au=\Delta ^{2}u$,
where $\mathscr{D}(A)$ incorporates the corresponding boundary conditions
(clamped, hinged, or free). It is useful to note that by elliptic regularity
\begin{equation*}
\mathscr{D}(A^{1/2})=\left\{
\begin{array}{cc}
H_{0}^{2}(\Omega ) & \text{clamped BC} \\
(H^{2}\cap H_{0}^{1})(\Omega) & \text{hinged BC} \\
(H^{2}\cap H_{0,\Gamma _{0}}^{2})(\Omega ) & \text{free BC}
\end{array}
\right.
\end{equation*}

It is important to note the total potential energy may not be positive, or
even not bounded from below. This is due to the presence of internal force $%
F_0$ which may drive the energy to $- \infty $. However, the presence of the von
Karman bracket in the model, along with appropriate regularity properties
imposed on $F_0$, assures that the energy is bounded from below. This can be
seen from the following lemma \cite{springer,ch-l-jde04}:

\begin{lemma}
\label{l:1} Let $u \in \mathscr{D}(A^{1/2} ) $, $p\in L_2(\Omega),$ and $F_0
\in H_0^1(\Omega) \cap H^{\theta}(\Omega) $, $\theta > 3 $. Then, $\forall
~\epsilon > 0 $ there exists $M\big(\epsilon,||p||, ||F_0||_{\theta}\big)=M_{\epsilon,p,F_0} < \infty $
such that in the clamped and hinged case
\begin{equation*}
||u||^2 \leq \epsilon \big( ||{\mathcal{A}}^{1/2} u||^2 + ||\Delta v(u) ||^2 %
\big) + M_{\epsilon, p, F_0}
\end{equation*}
and in the free case with $\beta > 0 $,
\begin{equation*}
||u||^2 \leq \epsilon \big( ||{\mathcal{A}}^{1/2} u||^2 + ||\Delta v(u) ||^2
+ \frac{\beta}{2} ||u||^4_{L_4(\Gamma) } \big) + M_{\epsilon,p,
F_0,\beta}
\end{equation*}
\end{lemma}

As a consequence we have the following bounds from below for the energy:

\noindent There exist positive constants $m,c,M,C$ such that
\begin{align}
-m+c\widehat{E}(t)& \leq \mathscr{E}(t)\leq M+C\widehat{E}(t)  \label{lowerb}
\\
-m+cE(t)& \leq \mathscr{E}(t)\leq h(E(t))
\end{align}%
where $h(s)$ denotes a continuous function.

\subsection{Motivation and Literature}

Well-posedness for von Karman's plate equation with interior and/or boundary
dissipation has been known for some time for smooth solutions in the case of
homogeneous \cite{ch-1} or inhomogeneous nonlinear boundary conditions \cite%
{springer,fhlt} and references therein. The issue of well-posedness for
`weak' (finite-energy) solutions is more recent \cite{springer,fhlt}. In
this paper, we are interested in homogeneous type boundary conditions and we
will be considering \textit{generalized} nonlinear semigroup solutions \cite%
{barbu,showalter} which also can be shown to be \textit{weak} variational
solutions. For a detailed and complete discussion regarding the wellposedness
and regularity of von Karman solutions the reader is referred to \cite%
{springer,koch}. In the context of this paper we will need the following
well-posedness result, which is contingent upon the recently shown
sharp regularity of the Airy Stress function in (\ref{airy}) \cite%
{fhlt,springer}:

\begin{theorem}
\label{wellp} With reference to problem \ref{plate}(C) with initial data $%
(u_0,u_1)\in \mathcal{H}_1$, or \ref{plate}(H) with initial data $%
(u_0,u_1)\in \mathcal{H}_2$, or \ref{plate}(F) with intial data $%
(u_0,u_1)\in \mathcal{H}_3$, there exists a \textit{unique} global solution
of finite-energy (i.e. $(u,u_t) \in C([0,T];\mathcal{H}_i)$ for $i=1,~2,~3$
resp., for any $T>0$). Additionally, $(u,u_t)$ depends continuously on $%
(u_0,u_1)\in \mathcal{H}_i$.
\end{theorem}

Thus, for any initial data in the finite energy space $(u_{0},u_{1})\in
\mathcal{H}$, there exists a well defined semiflow (nonlinear semigroup) $%
S_{t}(u_{0},u_{1})\equiv \big(u(t),u_{t}(t)\big)\in \mathcal{H}$ which
varies continuously with respect to the initial data in $\mathcal{H}$. The
domain of the corresponding generator ${\mathcal{A}}(u,v)\equiv \Big(v,-Au-d(%
\mathbf{x})g(v)+f_{V}(u)+p\Big)$ is given by $\displaystyle\mathscr{D}({%
\mathcal{A}})=\{(u,v)\in \mathscr{D}(A^{1/2})\times \mathscr{D}%
(A^{1/2});Au+d(\mathbf{x})g(v)\in L_{2}(\Omega )\}$. For initial data
taken in $\mathscr{D}(A)$, the corresponding solutions are regular and
remain invariant in $\mathscr{D}({\mathcal{A}})$ \cite{barbu,Pazy,showalter}. With an additional assumption that $g(s)$ is bounded polynomially at
infinity, one has $\mathscr{D}({\mathcal{A)}}\subset H^{4}(\Omega )\times
H^{2}(\Omega )$. Equipped with the regularity of the domain $\mathscr{D}({%
\mathcal{A}})$, one derives the energy identity for all regular solutions.
Due to the density of the embedding $\mathscr{D}({\mathcal{A}})\subset
\mathcal{H}$, monotonicity of the damping, and sharp regularity of the Airy
stress function (see Lemma \ref{t:3.4}) the same energy equality remains
valid for all generalized solutions corresponding to any boundary conditions
under consideration. Thus we have the energy identity for boundary
conditions (C), (H), or (F) satisfied for all generalized (semigroup)
solutions (complete details of this argument are given in \cite{springer}).

This equality reads: for all $0<s<t$, strong and generalized solutions $u$
to (\ref{plate}) satisfy
\begin{equation}
\mathscr{E}(t)+\int_{s}^{t}\int_{\Omega }d(\mathbf{x})g(u_{t})u_{t}=%
\mathscr{E}(s)  \label{Eident}
\end{equation}%
With the well-posedness of the semiflow established in Theorem 1.1, it is
natural to investigate long time behavior of the dynamical system generated
by (\ref{plate}). It is clear from (\ref{Eident}) that the essential
mechanism for dissipating the energy is the damping term $d(x)g(u_{t})$. In
the simplest possible scenario when $p=F_{0}=0$ the energy function ${%
\mathscr{E}}(t)$ is equivalent topologically to the norm of the phase space $%
\mathcal{H}$. Since ${\mathscr{E}}(t)$ is nonincreasing on the trajectories,
it becomes a Lyapunov function for the corresponding nonlinear dynamical system, whose only
equilibrium is the zero point. If one assumes that $d(x)>0,$ $a.e.$in $ \Omega $, it is well known that ${\mathscr{E}}(t)$ becomes a \textit{strict}
Lyapunov function and zero equilibrium is strongly stable. However, the
above condition imposed on $d(x)$ is \textit{not sufficient} to guarantee
uniform convergence to the equilibrium (this is also the case for \textit{%
linear} dynamics without the von Karman term). In order to secure uniform
convergence or, more generally, convergence to a compact attractor, a
stronger form of the damping is necessary. For example, $d(x)\geq c_{0}>0,$ $%
x\in \Omega $ and $g(s)=as,$ $a>0$, provides a classical model for which
uniform convergence to zero in the absence of external/internal forcing (or
more generally to an attractor) can be shown \cite{ch-0,ch-1,springer,lagnese}. The goal in this paper is to consider nonlinear damping of a reduced
essential support whereby the inequality $d(x)\geq c_{0}>0$ will be enforced
only in a small set $\omega \subset \subset \Omega$, while the dynamics will be forced by nontrivial sources $p,~F_{0}$. Existence of a
compact and possibly smooth finite dimensional attracting set for the
dynamics generated by (\ref{plate}) with boundary conditions (C), (H), or
(F) and geometrically constrained dissipation is of great physical interest.
Such a result is tantamount to asserting that the infinite dimensional,
non-smooth dynamics are asymptotically reduced to a \textit{smooth and
finite dimensional set.} While such a reduction is expected for dynamical
systems that exhibit some smoothing effects (e.g. parabolic-like) \cite%
{temam,milani,eden,raugel,babin,miranville,kalantarov}, it is a much less
evident phenomena in the case of hyperbolic-like dynamics, where the
`taking-off' of the dynamics produces no smoothing effect. The role of the
frictional damping in such a system is instrumental; in fact, it is the
induced friction that creates a stabilizing and asymptotically regularizing
effect on the evolution, ultimately reducing it to a compact set. On the
other hand it is well known that the hyperbolic-like dynamics can not be
stabilized by a compact feedback operator \cite{cbms} (and references
therein). This is due to the fact that instabilities in the system are
inherently infinite dimensional and the essential part of the spectrum can
not be dislodged by a compact perturbation. Thus, any effective damping cannot be compact (with respect to the phase space). The
above feature combined with (a) nonlinearity of the damping and (b) lack of
compactness of the nonlinear von Karman source makes the analysis of
long-time behavior for this class of systems challenging. In fact, critical
exponent nonlinearities and nonlinear dissipation are known to constitute
endemic difficulties in the study of hyperbolic-like systems \cite{fereisel}.

To orient the reader and to provide some perspective for the problem
studied, we shall briefly describe some of the principal contributions to
this area of research. A detailed account is given in \cite{springer}.

In the discussion of global attractors for von Karman evolution equations,
we must distinguish between two types of dynamics for the problem: (a) the
rotational case (as addressed above) when the term $-\gamma\Delta
u_{tt},~\gamma>0$ is added to the LHS of (\ref{plate}) and (b) nonrotational ($
\gamma=0$). In case (a), we note that the von Karman nonlinearity (in the
finite energy topology) is \textit{compact}, which considerably simplifies
the analysis of long-time dynamics. In the latter case (b) (which we
consider here), a very different type of analysis is needed. Here, we shall
focus on part (b) only. In fact, the very first contribution to this problem
is a pioneering paper \cite{ch-1} where the existence of \textit{weak}
attractors with a linear, fully supported damping was demonstrated. Later on,
owing to new results on the regularity of Airy's stress function \cite%
{fhlt,springer}, \textit{weak} attractors were improved to \textit{strong}
attractors, and the restriction of linear damping was removed in order to
allow nonlinear, monotone damping \cite{chlJDE04}. In order to incorporate
fully nonlinear interior damping, \cite{chlJDE04} assumes that the
dissipation parameter is sufficiently \textit{large}. This restriction was
later removed in \cite{kh}, whose paper introduces a very clever way of
bypassing a lack of compactness and replacing it with an ``iterated
convergence" trick. Further studies of the attractor (including properties
such as dimensionality and smoothness) in the fully nonlinear setup, without
``size" restrictions imposed on the parameters, are presented in \cite%
{ch-l-jde07} and in monograph form in \cite{ch-l,springer}.

It should be noted that the results described above pertain to the interior
and \textit{fully supported} dissipation. The situation is much more
delicate when the dissipation is \textit{geometrically constrained}, where
the essential support of the damping is localized to a \textit{subset} of
the spatial domain $\Omega$. In that case, the issue of propagating the
damping from one area to another becomes the critical one. While this sort
of problems has been previously studied in the context of stabilization to
equilibria \cite{lagnese,horn,horn1}, the estimates needed for attractors
are much more demanding. Previous methods developed in the context of
stabilization no longer apply. Some long-time behavior results with boundary
damping are presented in \cite{ch-l-jde04,ch-l-jde07}, wherein nonlinear
dissipation on the boundary acting via \textit{free} boundary conditions is
considered. These works, however, impose the rather stringent geometric
restrictions of the entire boundary being star-shaped. Such
restrictions are removed in \cite{springer}, where dissipation via hinged
boundary conditions is considered; however this is done at the expense of
limiting the class of dissipation to those of \textit{linearly bounded}
type. This restriction is needed since the elimination of the geometric
condition is achieved via microlocal estimates \cite{sharp}, which in
turn force velocity dependent nonlinear terms to be linearly bounded.

This brings us to the main contribution of the present manuscript. Our goal
is to show that the fully \textit{nonlinear} damping with essential support
in an arbitrarily small layer near the boundary provides not only the
existence of compact attractors but also desirable properties such as
smoothness and finite dimensionality. Thus the original hyperbolic-like
non-smooth flow is asymptotically reduced to smooth and finite dimensional
dynamics. The result is valid for \textit{all types of boundary conditions}
with \textit{geometrically constrained dissipation}, which can be nonlinear
of \textit{any polynomial growth} at infinity and with \textit{no
restriction on the size} of the damping parameter.

We obtain this result by proving that the dynamics are \textit{quasi-stable}
- a concept introduced in \cite{ch-l} and \cite{springer}. The ability to
show quasi-stability is dependent upon: (a) a new method of localization of
multipliers that allows smooth propagation of the damping from the boundary
collar into the interior (even in the presence of boundary conditions (free)
that do not comply with the Lopatinski conditions \cite{sakamoto}) and (b)
``backward" smoothness of trajectories from the attractor -
a method used also in \cite{ch-l-jde07} and in \cite{daniel1,bucci} - the
latter in the context of geometrically constrained dissipation for wave
dynamics.

Lastly, we would like to note that while some of the methods
developed for \textit{boundary} dissipation \cite{springer,ch-l-jde07} can
also be used in the case of partially localized dissipation and
Dirichlet - clamped boundary conditions, this is not the case with Neumann
type (free) boundary conditions which violate strong Lopatinski \cite{sakamoto} condition. In this latter case, propagation of the damping from
the boundary layer via boundary damping estimates is obstructed by the well
known lack of sufficient regularity (the absence of so called ``hidden" regularity \cite%
{lions}) of boundary traces corresponding to the linear model \cite{sharp}.
Our aim in this paper is to develop a method which is effective for all kind
of boundary conditions and does not depend on hidden regularity, where the
latter restricts the analysis to Lopatinski type of models. The key element
for this are suitably localized multipliers estimates.

\subsection{Statement of Results}

Equipped with well-posedness of finite energy and regular solutions
corresponding to (\ref{plate}) under one of the boundary conditions (C),
(H), or (F), we are now ready to state our main results pertaining to long
time behavior of solutions. In order to do this, we shall introduce the
following unique continuation condition, denoted $UC$. We say that the
system satisfies the $UC$ property iff the following implication is valid for
any weak solution $(u,u_{t})$ to (\ref{plate}): There exists $T>0$ such
that
\begin{equation*}
u_{t}=0\text{ a.e }~\text{in}~\text{supp}~d\times (0,T)\Rightarrow ~u_{t}=0~%
\text{a.e in}~\Omega \times (0,T).
\end{equation*}%
It is clear that the $UC$ property holds if $d(\mathbf{x})>0~\text{a.e. in}%
~\Omega $.

\begin{remark}
It is worth noting that due to the non-local character of von Karman
nonlinearity, the unique continuation property for the von Karman plate is
not fully understood. A now classical set of tools developed for plate equations
and based on Carleman estimates \cite{tataru,albano,isakov,eller} do not
apply. The non-locality of the von Karman bracket prevents propagation
across the entire domain of \textit{weak} damping localized to a small set.
Therefore, we have the question: \textit{if the damping in the equation (represented
by $d(\mathbf{x})u_{t}$) is zero in an open set of positive measure inside
of $\Omega $, does this imply that the solution $u$ must also be 0 in $\Omega$?}; it remains open. In relation to our analysis here, if the general unique
continuation property holds for the von Karman plate, then it immediately
strengthens our result by allowing $d(\mathbf{x})$ to vanish away from an
open collar of the boundary. However, at present, the best we can state is a
sufficient condition, namely that $d(\mathbf{x})>0$ a.e. in order to satisfy
the $UC$ property.
\end{remark}

\noindent In addition, we shall assume validity of an \textit{asymptotic}
growth condition \textit{from below} imposed on $g(s) $. Such condition is
typical \cite{lagnese} and necessary in order to obtain uniform decay rates
of solutions in hyperbolic-like dynamics. It allows control of the kinetic
energy for large frequencies.

\begin{assumption}
\label{g} There exist positive constants $0<m \leq M < \infty $ and a
constant $p \geq 1 $ such that
\begin{equation*}
m \leq g^{\prime}(s) \leq M |s|^{p }, ~~ |s| \geq 1
\end{equation*}
\end{assumption}

\noindent We now state the primary result in this treatment.

\begin{theorem}
\label{t:1} Take Assumption \ref{g} to be in force. Let $\text{supp}%
~d\supset \omega $ and $d(\mathbf{x}) \ge \alpha_0 >0$ in $\omega$, where $%
\omega \subset \subset \Omega$ is any full collar near the boundary $\Gamma$%
. Then for all generalized solutions corresponding to solutions with initial
data $||(u_0,u_1)||_{\mathcal{H}}\leq R$, there exist compact attractor $%
\mathbf{A}_{R}\in \mathcal{H}$. If, in addition, the \textit{UC} property
holds, then said attractor is global, i.e $\mathbf{A }_{R}=\mathbf{A}$
for all $R>0$.
\end{theorem}

\begin{remark}
The $UC$ property is needed in order to construct a \textit{strict} Lyapunov
function for the plate. In the absence of this Lyapunov function, the obtained results are of local character
- as in the first part of Theorem \ref{t:1}. The $UC$ property allows us to conclude
that local attractors coincide with a global one.
\end{remark}

\begin{theorem}
\label{t:2} In addition to Assumption 1 and the \textit{UC} property, assume
that there exists $m,M>0$, and $\gamma<1 $ such that $0<m\leq g^{\prime
}(s)\leq M[1+sg(s)]^{\gamma}$, for all $s \in \mathbb{R} $. Then,

(a) the attractor $\mathbf{A} $ is regular, which is to say $\mathbf{A}%
\subset H^{4}(\Omega )\times H^{2}(\Omega )$ is a bounded set in that
topology.

(b) The fractal dimension of $\mathbf{A} $ is finite.
\end{theorem}

\begin{remark}
If we consider $g(s) = |s|^{p+1}$, then one can show that $\gamma = \dfrac{p%
}{ p+2}$ satisfies the above condition.
\end{remark}

There are three main difficulties/novelties pertaining to the proof of the
results stated above:

(a) The nonlinear source is of critical exponent (lack of compactness).

(b) The essential damping is geometrically constrained to a small subset $%
\omega$.

(c) The damping is genuinely non-linear (any polynomial growth at the
infinity is allowed).

\noindent These three difficulties are well-recognized in the context of studying long
time behavior of hyperbolic-like systems where there is no inherent
smoothing mechanism present in the model. In order to provide some
perspective, it helps to add that geometrically constrained damping forces
to use higher order multipliers which become \textit{supercritical} when
dealing with energy terms and nonlinear critical terms. Thus, any successful
approach must rely on suitable cancellations, which must be uncovered for
the specific dynamics in question.

Similar issues appear when dealing with nonlinear damping. The damping term
must be critical (in hyperbolic dynamics) in order to be effective (we
recall that the essential spectrum of an operator can not be altered by a
compact perturbation). The property of monotonicity of the problem does help
when dealing with a single solution at the energy level. However, when
dealing with long-time behavior, the protagonist is not a single solution
but the difference of two solutions. In the study of the corresponding
dynamics at the non-energetic levels (resulting from multipliers),
monotonicity is destroyed. There is a ``spillover" of the noncompact (in
fact, supercritical) damping that must be absorbed. For this issue,
different mechanisms need to be discovered (e.g. backward smoothness of
trajectories, compensated compactness, etc).

While recent developments in the field provide tools enabling us to handle a
combination of \textit{any two} of the difficulties listed above, the
\textit{inclusion of the third} prevents us from utilizing existing
mathematical technology. The principal contribution of this treatment is to
develop method which is capable of dealing with all three aforementioned
difficulties simultaneously. The main ingredients of this new approach are
(i) a localization method which allows us to show propagation of the damping
without any requiring that the Lopatinski condition be satisfied, and (ii)
backward smoothness of trajectories from the attractor with geometrically
localized dissipation.

We conclude this section by listing few problems that are of interest to
pursue and still open. \vskip.1cm \noindent \textbf{(1) $C^{\infty}$
smoothness of attractors in the presence of nonlinear damping}

Regarding the first item, $C^{\infty}$ smoothness of an attractor can be
proved under certain restrictions on nonlinear damping by methods developed
in \cite{springer} and also in an influential paper \cite{guidaglia}. The
treatment of a fully nonlinear and monotone dissipation is still not fully
understood. \vskip.1cm \noindent \textbf{(2) Damping restricted to a portion
of an open collar}

Secondly, dissipation localized to part of the collar could be considered by
assuming certain geometric conditions imposed on the uncontrolled part of
the collar. Certain ideas presented in \cite{daniel1,bucci} should prove
useful. \vskip.1cm \noindent \textbf{(3) The $UC$ property for a larger class of
dampings}

Lastly, the third item is an wide open and important problem. Other mild forms of dissipation - such as viscoelastic weak damping - are also
valid alternatives, however, a full understanding of this problem would
require the development of a substitute for Carleman estimates, which are
not applicable due to localized structure of the von Karman bracket. The
associated problem is related to controlling low frequencies - an endemic
problem principally associated with strong stability \cite{cbms,jde}. It is
worth noting that this problem is non-existent in the case of other
nonlinear plates, such as semilinear plates with local nonlinearities or
even Berger's model \cite{berger}, where some forms of Carleman's estimates
do apply \cite{albano,tataru}.

\section{Long-time Behavior of Dynamical Systems}

\indent In this manuscript we will make ample use dynamical systems
terminology (see \cite{babin,miranville,raugel,ch-0,springer}); let $(%
\mathcal{H},S_t)$ be a dynamical system with $\mathscr{N}\equiv \{x\in%
\mathcal{H}:S_tx=x~\text{ for all }~ t \ge 0\}$ the set of its stationary
points.

We say that a dynamical system is \textit{asymptotically compact} if there
exists a compact set $K$ which is uniformly attracting: for any bounded set $
D\subset \mathcal{H}$ we have that
\begin{equation}  \label{dist}
\lim_{t\to+\infty}d_{\mathcal{H}}\{S_t D|K\}=0
\end{equation}
in the sense of the Hausdorff semidistance.

$(\mathcal{H},S_t)$ is said to be \textit{asymptotically smooth} if for any
bounded, forward invariant $(t>0) $ set $D$ there exists a compact set $K
\subset \overline{D}$ such that (\ref{dist}) holds. An asymptotically smooth
dynamical system should be thought of as one which possesses \textit{local
attractors}, i.e. for a given ball $B_R(x)$ of radius $R$ in the space $
\mathcal{H}$ there exists a compact attracting set in the closure of $B_R(x)$%
, however, this set need not be uniform with respect to $R$ or $x \in
\mathcal{H}$.

A \textit{global attractor} $\mathbf{A}$ is a closed, bounded set in $%
\mathcal{H}$ which is invariant (i.e. $S_t\mathbf{A}=\mathbf{A}$ for all $%
t>0 $) and uniformly attracting (as defined above).

A \textit{strict Lyapunov function} for $(\mathcal{H},S_t)$ is a functional $%
\Phi$ on $\mathcal{H}$ such that (a) the map $t \to \Phi(S_tx)$ is
nonincreasing for all $x \in \mathcal{H}$, and (b) $\Phi(S_tx)=\Phi(x)$ for
all $t>0$ and $x \in \mathcal{H}$ implies that $x$ is a stationary point of $%
(\mathcal{H},S_t)$. If the dynamical system has a strict Lyapunov function,
then we say that $(\mathcal{H},S_t)$ is \textit{gradient}.

In the context of this paper we will use a few keys theorems (which we now
formally state) to prove the existence of the attractor. (For proofs and
references, see \cite{springer} and references therein.) First, we address
attractors for gradient systems and characterize the attracting set:

\begin{theorem}
\label{gradsmooth} Suppose that $(\mathcal{H},S_t)$ is a gradient,
asymptotically smooth dynamical system. Suppose its Lyapunov function $\Phi
(x)$ is bounded from above on any bounded subset of $\mathcal{H}$ and the
set $\Phi _{R}\equiv \{x\in \mathcal{H}:\Phi (x)\leq R\}$ is bounded for
every $R$. If the set of stationary points for $(\mathcal{H},S_t)$ is
bounded, then $(\mathcal{H},S_t)$ possesses a compact global attractor $%
\mathbf{A}$ which coincides with the unstable manifold, i.e.
\begin{equation*}
\mathbf{A}=\mathscr{M}^{u}(\mathscr{N})\equiv \{x\in \mathcal{H}:~\exists
~U(t)\in \mathcal{H},~\forall ~t\in \mathbb{R}~\text{ such that }~U(0)=x~%
\text{ and }~\lim_{t\rightarrow -\infty }d_{\mathcal{H}}(U(t)|\mathscr{N}%
)=0\}.
\end{equation*}
\end{theorem}

Secondly, we state a useful criterion (inspired by \cite{kh}) which reduces
asymptotic smoothness to finding a suitable functional on the state space
with a compensated compactness condition:

\begin{theorem}
\label{psi} Let $(\mathcal{H},S(t))$ be a dynamical system, $\mathcal{H}$
Banach with norm $||\cdot||$. Assume that for any bounded positively
invariant set $B \subset \mathcal{H}$ and for all $\epsilon>0$ there exists
a $T\equiv T_{\epsilon,B}$ such that
\begin{equation*}
||S_Tx_1 - S_Tx_2||_{\mathcal{H}} \le
\epsilon+\Psi_{\epsilon,B,T}(x_1,x_2),~~x_i \in B
\end{equation*}
with $\Psi$ a functional defined on $B \times B$ depending on $\epsilon, T,$
and $B$ such that
\begin{equation*}
\liminf_m \liminf_n \Psi_{\epsilon,T,B}(x_m,x_n) = 0
\end{equation*}
for every sequence $\{x_n\}\subset B$. Then $(\mathcal{H},S_t)$ is an
asymptotically smooth dynamical system.
\end{theorem}

In order to establish both smoothness of the attractor and finite
dimensionality a stronger estimate on the difference of two flows is needed.

\begin{theorem}
\label{t:FD} Let $x_{1},x_{2}\in B\subset \mathcal{H}$ where $B$ is a
forward invariant set for the flow $S_{t}x_{i} $.
Assume that the following inequality holds for all $t>0$ with positive
constants $C_{1}(B),C_{2}(B),\omega _{B}$
\begin{equation}
||S_{t}x_{1}-S_{t}x_{2}||_{\mathcal{H}}^{2}\leq C_{1}(B)e^{-\omega
_{B}t}||x_{1}-x_{2}||_{\mathcal{H}}^{2}+C_{2}(B)\max_{\tau \in \lbrack
0,t]}|| S_{\tau} x_1 - S_{\tau} x_2 ||_{\mathcal{H}_1}^{2}  \label{quasi}
\end{equation}%
where $\mathcal{H} \subset\mathcal{H}_1 $ is compactly embedded. Then the
attractor $\mathbf{A}$ associated with the flow $S_{t}$ posesses the
following properties: \newline
(a) The fractal dimension of $\mathbf{A}$ is finite. \newline
(b) For any $x \in \mathbf{A} $ one has $\dfrac{d}{dt}\big( S_t x \big)\in
C(R, \mathcal{H} ) $.
\end{theorem}

\begin{remark}
The estimate in (\ref{quasi}) is often referred to as a ``quasistability"
estimate. It reflects the fact that the flow can be stabilized exponentially
to a compact set. Alternatively, we might say that the flow is exponentially
stable, modulo a compact perturbation (lower order terms). We assert that
the lower order terms being quadratic is important for the validity of
Theorem \ref{t:FD}.
\end{remark}

The proof of Theorem \ref{t:FD} employs the idea of ``piecewise" trajectories
introduced in \cite{malek,prazak}. This allows to generalize previous
criteria for finite-dimensionality \cite{babin,temam,raugel,eden} by
reducing the problem to validity of quasistable estimate.

\subsection{Approach and Outline of the Paper}

\indent To show our main result on the existence of the global attractor for
(\ref{plate}) with boundary conditions (C), (H), or (F) we make use of the
theorems above. First, we note that in the case of any boundary conditions,
the von Karman system in (\ref{plate}) is gradient with Lyapunov function $%
\mathscr{E}(t)$ (under the assumption that the $UC$ property is satisfied,
e.g. in the case that $d(\mathbf{x})>0$ a.e. in $\Omega$). We refer to \cite%
{springer} for the details. Moreover, the set of stationary points for the
dynamical system generated by (\ref{plate})(C), (\ref{plate})(H), or (\ref%
{plate})(F) is bounded. This latter fact follows from (\ref{lowerb}) (see
\cite{springer}). Hence we are in a position to use Theorem \ref{gradsmooth}
if we can obtain an inequality of the form in Theorem \ref{psi} to show
asymptotic smoothness of the system. We will analyze $z$, taken to be the
difference of strong solutions, and make use of the linear energy $%
\displaystyle E_z(t)=||\Delta z||^2 + ||z_t||^2$ (then, via a standard
limiting procedure obtain our estimate for generalized solutions as well);
this estimation will produce our functional $\Psi$ in Theorem \ref{psi}. Our
main tool in estimating $E_z(t)$ will be the use of two multipliers: $%
f_1 z$ and $h \cdot \nabla \big((f_2)z)\big)$, where $h$ will be a suitably
chosen $C^2$ vector field and $f_i$ are appropriate localization functions.

First, we perform multiplier analysis as generally as possible, without
imposing boundary conditions. Later on, we shall use boundary conditions
(either clamped, or hinged or free) in order to obtain the smoothness
inequality in Theorem \ref{psi}.

After establishing the existence of the attractor, we proceed to show that
it has additional regularity than that of the state space, and also that it
has finite fractal dimension. The ultimate goal is to prove a
``quasistability" estimate for the difference of general trajectories and
apply abstract Theorem \ref{t:FD}, however doing so directly in this case is
difficult. Instead, we will prove a modified quasistability estimate (via
similar methods in the asymptotic smoothness calculation) which applies to
the difference of two terms along the same trajectory (i.e. the difference
of $u(t+h)-u(t)$). This term is substantially easier to analyze, since we
have continuity in $t$ in $\mathcal{H}$ and both $u(t+h)$ and $u(t)$
converge to the \textit{same} point of equilibrium for $t \to \infty$ and $t
\to -\infty$. Our modified quasistability estimate hinges upon the bounding
terms being quadratic, so upon division by $h$ and taking $h \searrow 0$ we
can show additional regularity of elements from the attractor for
sufficiently negative times $T<<0$. Proving this will depend upon a
trajectory being `close' to a point of equilibrium, and hence yielding
`smallness' of the velocity of the solution. Propagating this regularity
forward via the dynamical systems property will then allow us to show the
additional regularity of the attractor. We will then proceed in a standard
fashion to show that this additional regularity of the attractor yields the
true and sought after quasistability estimate in Theorem \ref{t:FD}, which
will produce the finite fractal dimension of the attractor.

\section{Asymptotic Smoothness}

In this section we prove that the dynamical system generated by (\ref{plate}%
) is asymptotically smooth. We will refrain from imposing boundary
conditions until absolutely necessary in the hope of unifying the treatment
of (C), (H), and (F).

Note that the new variable $z=u-w$, where $%
(u(t),u_{t}(t))=S_{t}(u_{0},u_{1}) $, $(w(t),w_{t}(t))=S_{t}(w_{0},w_{1})$
are solutions to (\ref{plate}) with initial data taken in bounded set in $%
B\subset \mathcal{H}$. On the strength of Lemma 1.1 and (\ref{lowerb}) we
may assume that there exists $R>0$ such that
\begin{equation}
||(u(t),u_{t}(t)||_{\mathcal{H}}\leq R,~~||(w(t),w_{t}(t)||_{\mathcal{H}%
}\leq R,\text{ \ }t>0  \label{R}
\end{equation}%
The difference of two trajectories $z=u-w$ solves the following PDE:
\begin{align}
z_{tt}+\Delta ^{2}z+& \mathcal{G}(z)+\mathcal{F}(z)=0~\text{ in }~Q,
\label{diff} \\
z(0)=u_{0}-& w_{0};~z_{t}(0)=u_{1}-w_{1}  \notag
\end{align}%
where
\begin{equation*}
\mathcal{F}(z)\equiv -(f_{V}(u)-f_{V}(w)),~and~\mathcal{G}(z)\equiv d(%
\mathbf{x})(g(u_{t})-g(w_{t}))
\end{equation*}%
The above evolution is equipped with appropriate boundary conditions (C),
(H), or (F) which will be specified later.

\subsection{Multipliers}

\indent Ultimately, we will need a pointwise bound (in time) on the
functional $E_z(t)$ as defined above. To achieve this bound, we will employ
multiplier methods based on specially chosen cut-off functions $\lambda$ and
$\mu$. These functions are taken to be $C^{\infty}(\Omega)$. Later, we will
choose the supports of the derivatives of $\lambda$ and $\mu$ to be
contained in the damping region $\omega$, where the damping $g(u_t)$ is
effectively localized; the cut-off functions will be chosen in this way so
as to reconstruct the full energy $E_z(t)$ via the multipliers, bounded in
terms of the damping. However, for now, we can consider $\text{supp}~\lambda
\subset \Omega$ to be arbitrary.

We define the variables $\phi =\lambda z$ and $\psi =\mu z$. The use of the
cut-off functions will produce commutators active in the regions of $\omega $
where the cut-off functions are non-constant. Lastly, we will make use of
the following notational conventions. First, to describe (a) lower order
terms:
\begin{equation*}
\displaystyle l.o.t.^{f}\equiv \sup_{\lbrack 0,T]}||f(t)||_{2-\eta }^{2},~~~%
\displaystyle l.o.t._{1}^{f}\equiv \sup_{\lbrack 0,T]}||f(t)||_{2-\eta },
\end{equation*}%
where $0<\eta <1/2,$ and (b) boundary terms: $\displaystyle B.T.^{f}=\Big\{%
\Delta f\partial _{\nu }f-\partial _{\nu }(\Delta f)f\Big\}$

\begin{remark}
We note that the use of different notations for lower order terms is
necessary in the handling of dissipation estimates. Specifically, we must
treat the dissipation terms differently when dealing with asymptotic
smoothness type estimates, and the estimates which will ultimately yield the
quasistability estimate.
\end{remark}

\subsubsection{$\protect\phi$ Multiplier}

Let $P$ and $Q$ be two differential operators. We will make use of the
commutator symbol given by
\begin{equation*}
\lbrack P,Q]f=P(Qf)-Q(Pf),
\end{equation*}%
We shall work with smooth solutions guaranteed by Theorem \ref{wellp}.
Multiplying the PDE in (\ref{diff}) by $\lambda $ we arrive at
\begin{equation*}
\phi _{tt}+\Delta ^{2}\phi +\lambda \mathcal{G}(z)+\lambda \mathcal{F}%
(z)=[\Delta ^{2},\lambda ]z.
\end{equation*}%
Now, we employ the multiplier $\phi $. This is an equipartition multiplier
which allows us to reconstruct the difference between the potential and
kinetic energies. The following Green's identities are available \cite%
{lagnese} for sufficiently smooth functions $z$ and $\phi $:
\begin{equation*}
\begin{cases}
\displaystyle\int_{\Omega }\Delta ^{2}z\phi =\int_{\Omega }\Delta z\Delta
\phi +\int_{\Gamma }(\partial _{\nu }\Delta z\phi -\Delta z\partial _{\nu
}\phi ),~\text{~ clamped and hinged B.C } &  \\[0.3cm]
\displaystyle\int_{\Omega }\Delta ^{2}z\phi =a(z,\phi )+\beta \int_{\Gamma
_{1}}z^{3}\phi +\int_{\Gamma _{1}}\big(\mathcal{B}_{2}z\phi -\mathcal{B}%
_{1}z\partial _{\nu }\phi \big), & \text{~ free B.C }%
\end{cases}%
\end{equation*}%
Using the first formula for clamped or hinged boundary conditions yields:
\begin{equation}
\int_{Q}\big\{|\Delta \phi |^{2}-|\phi _{t}|^{2}\big\}=~\int_{Q}[\Delta
^{2},\lambda ]z\phi -\int_{Q}\lambda \big\{\mathcal{G}(z)+\mathcal{F}(z)%
\big\}\phi +\int_{\Sigma }\big\{\Delta \phi \partial _{\nu }\phi -\partial
_{\nu }(\Delta \phi )\phi \big\}-(\phi _{t},\phi )\big|_{0}^{T}  \label{3.3}
\end{equation}%
Making use of standard splitting and Sobolev embeddings, we arrive at
\begin{equation}
\int_{0}^{T}\big\{||\Delta \phi ||^{2}-||\phi _{t}||^{2}\big\}\leq
~\int_{\Sigma }B.T.^{\phi }+\int_{Q}([\Delta ^{2},\lambda ]z)\phi
+\int_{Q}\lambda \big\{\mathcal{G}(z)+\mathcal{F}(z)\big\}\phi +C(E(T)+E(0))
\tag{3.4}  \label{3.4}
\end{equation}%
In the case of free boundary conditions, the equipartition of energy takes
the form
\begin{equation*}
\int_{0}^{T}\big\{a(\phi ,\phi )+\beta |\phi |_{L_{4}(\Gamma )}^{4}-||\phi
_{t}||^{2}\big\} \leq \int_{\Sigma _{1}}(\mathcal{B}_{1}\phi \phi -\mathcal{B%
}_{2}\phi \partial _{\nu }\phi )+\int_{Q}([\Delta ^{2},\lambda ]z)\phi
\end{equation*}
\begin{equation}
+\int_{Q}\lambda \big\{\mathcal{G}(z)+\mathcal{F}(z)\big\}\phi +C(E(T)+E(0))
\tag{3.5}  \label{3.5}
\end{equation}
We note for all boundary conditions (C), (H), the boundary terms $B.T.^{\phi
}\equiv 0$. In the free case (F) we have $\mathcal{B}_{1}\phi =0,~\mathcal{B}%
_{2}\phi =2\beta \phi uw$ where the latter term contributes a lower order
term to the estimate.

To continue with our observability estimation, we must explicitly bound
the commutator
$\displaystyle\int_{Q}[\Delta ^{2},\lambda ]z \phi.$ Purely algebraic
calculations give
\begin{align}
\lbrack \Delta ^{2},\lambda ]f=&\Delta ^{2}(\lambda f) -\lambda \Delta ^{2}f
\notag \\
=&(\Delta ^{2}\lambda) f +2\Delta \lambda \Delta f +2\big(\nabla \lambda,
\nabla (\Delta f)\big) +2\big( \nabla (\Delta \lambda),\nabla f) +2\Delta
(\nabla \lambda \nabla f)  \tag{3.6}
\end{align}
The calculation above implies that the commutator $[\Delta ^{2},\lambda ] $
is a differential operator of order three. In order to exploit this in the
calculations with the energy, we need to reduce the order of differential
operator acting on a solution via integration by parts. This is done below.

This computation makes sole use of Green's theorem. For the sake of
exposition, we do not impose any boundary conditions:
\begin{equation}  \label{sym1}
\int_{\Omega }\nabla\Delta u\big(\phi\nabla \lambda)=-\int_{\Omega }(\Delta
u)\text{div}(\phi \nabla \lambda)+\int_{\Gamma }(\phi\Delta u)\nabla \lambda
\cdot \nu  \tag{3.7}
\end{equation}

\begin{equation}  \label{sym2}
\int_{\Omega }\Delta (\nabla \lambda \nabla u)\phi =-\int_{\Omega }\nabla
(\nabla \lambda \nabla u)\nabla \phi +\int_{\Gamma }\partial _{\nu }(\nabla
u\nabla \lambda )\phi  \tag{3.8}
\end{equation}



Note that due to the fact that the support of $\nabla \lambda $ is away from
the boundary, all of the boundary terms in the above expressions (\ref{sym1}%
) and (\ref{sym2}) will vanish. Moreover,
\begin{equation}  \label{commest}
\Big|\int_{\Omega }\nabla \lambda \nabla \Delta u\phi \Big|+\Big|%
\int_{\Omega }\Delta (\nabla \lambda \nabla u)\phi \Big|\leq C_{\lambda
}||u||_{2}||\phi ||_{1}  \tag{3.9}
\end{equation}
Hence to conclude our $\phi $ multiplier estimate, we have the following
technical lemma:

\begin{lemma}[Preliminary $\phi $ Estimate]
\label{phiest}Let $\phi \equiv \lambda z$, as
defined above, where $z$ solves(\ref{diff}) with boundary conditions (C) or
(H). Then, there exists $0 < C < \infty $ such that
\begin{equation}
\int_{0}^{T}\big\{||\Delta \phi ||^{2}-||\phi _{t}||^{2}\big\} \leq
~C(T,\lambda )l.o.t.^{z}+\int_{Q}\lambda \big\{\mathcal{G}(z)+\mathcal{F}(z)%
\big\}\phi +C(E_{z}(T)+E_{z}(0))  \tag{3.10}
\end{equation}%
In the free case (F)
\begin{equation}
\int_{0}^{T}\big\{a(\phi ,\phi )+\beta \int_{\Gamma }\phi ^{4}-||\phi
_{t}||^{2}\big\}\leq ~C(T,\lambda , R)l.o.t.^{z}+\int_{Q}\lambda \big\{%
\mathcal{G}(z)+\mathcal{F}(z)\big\}\phi +C(E_{z}(T)+E_{z}(0))  \tag{3.11}
\end{equation}
\end{lemma}

\begin{proof}
Taking into account (\ref{commest}) in (\ref{3.4}), we have
\begin{align}
\int_0^T \big\{||\Delta \phi||^2-||\phi_t||^2\big\} \le&
~C(T,\lambda)l.o.t.^z +\int_Q\lambda\big\{ \cG+\cF\big\}\phi +
\int_{\Sigma} BT^{\phi}
 \nonumber \\
&~+\int_{\Sigma}\Big\{\Dn(\Delta(\lambda z))(\lambda
z)-\Delta(\lambda z)\Dn(\lambda z)-\Dn(\Delta z)\lambda^2z
\nonumber \\&~+2\lambda z(\Delta z)\Dn z+\lambda^2(\Delta z)\Dn z
\Big\} +C(E(T)+E(0)) \nonumber \hspace{3cm}(3.12)
\end{align}
Taking into consideration  boundary conditions (C) or (H) in (\ref{3.4}), noting that $ B.T^{\phi} =0 $ and accounting for the fact that
the  boundary  terms  resulting from the commutators vanish
  leads  to the first statement in the Lemma.
 Calculations in the free case are analogous, and result from
 (\ref{3.5}) and    $ \mathcal{B}_1 \phi =0 ,~ \mathcal{B}_2 \phi
=  2\beta \phi u w $,
where the latter term contributes a lower order term to the estimate:
$$\Big|\int_{\Gamma_1} \mathcal{B}_2 \phi \phi \Big|
\leq  2 \beta \int_{\Gamma_1} |\phi|^2 |u||w|
\leq   2\beta R^2 ||\phi||^2_1  \leq C(R)  l.o.t^z $$

\end{proof}

\subsubsection{Multiplier 2: $h\cdot \protect\nabla \protect\psi$}

For the first part of this section, we specify only that $\text{supp}~\mu
\cap \Gamma =\emptyset $; otherwise, we keep $\mu $ as general as possible,
specifying it at the last possible moment. Additionally, define a set $%
M\equiv \text{supp}~\nabla \mu =\overline{\{x\in \Omega \big|\mu \not\equiv
\text{constant}\}}$. Now, if we multiply (\ref{diff}) by $\mu ,$ and recall
that $\psi \equiv \mu z$, we obtain%
\begin{equation*}
\psi _{tt}+\Delta ^{2}\psi +\mu \mathcal{G}(z)+\mu \mathcal{F}(z)=[\Delta
^{2},\mu ]z
\end{equation*}%
where $\mathcal{G}(z)=d(\mathbf{x})\left( g(u_{t})-g(w_{t})\right) $ and $%
\mathcal{F}(z)=-(f_{V}(u)-f_{V}(w))$, as before. We now make use of the
multiplier $h\cdot \nabla \psi $, which we write as $h\nabla \psi $
henceforth; there are various choices for the vector field $h$,
situationally dependent, however here we need only take $h=\mathbf{x}-%
\mathbf{x}_{0}\in \mathbb{R}^{2}$ in order to obtain control on the
potential energy of the plate. Now, as in the previous section, we multiply
the last equality by our multiplier and use Green's Theorem to obtain
\begin{equation*}
\int_{Q}\left( |\psi _{t}|^{2}+|\Delta \psi |^{2}\right) \leq
~C(E_{z}(T)+E_{z}(0))+\int_{Q}\mu \mathcal{G}(z)(h\nabla \psi )+\int_{Q}\mu
\mathcal{F}(z)(h\nabla \psi )+\int_{Q}[\Delta ^{2},\mu ]z(h\nabla \psi )
\end{equation*}%
By explicitly writing out the commuator, and taking into account the support
of $\nabla \mu $, upon splitting we obtain:
\begin{equation}
\int_{Q}[\Delta ^{2},\mu ]z(h\nabla \psi )=~\int_{0}^{T}\int_{M}[\Delta
^{2},\mu ]z(h\nabla \psi )\leq C(\mu )\int_{0}^{T}\int_{M}|\Delta
z|^{2}+C(T,\mu )l.o.t.^{z}  \tag{3.13}
\end{equation}

\noindent Now, at this point we specify the specific structure of the
supports for $\lambda$ and $\mu$ (which up to now have been general).
The following picture illustrates our choice for these supports and their
relationship to the damping region $\omega$:

\begin{center}
\includegraphics[scale=1]{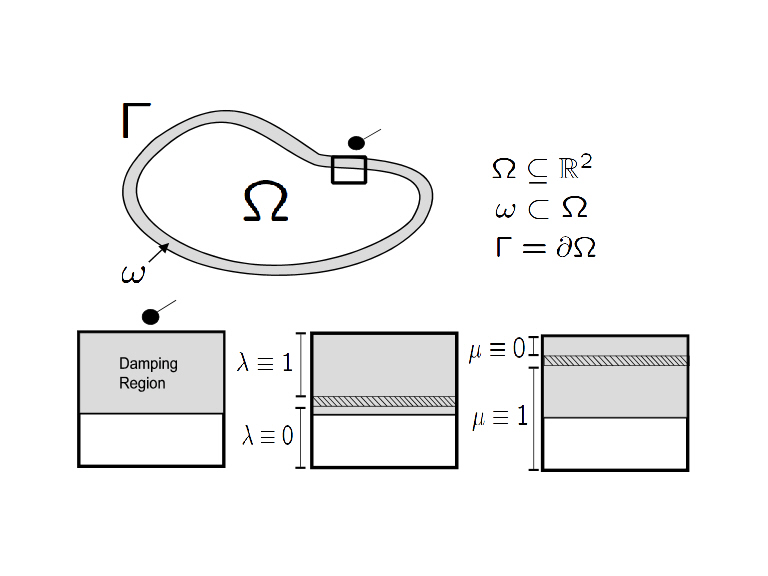}
\end{center}

We emphasize that (a) the set $M\subset \text{supp}~\lambda $ and (b) $\text{%
supp}~\lambda $ and $\text{supp}~\mu $ overlap inside the damping region $%
\omega $ and that $\text{supp}~\lambda \cup \text{supp}(\mu )=\Omega $.
Since $M\subset \{x\in \Omega :\lambda (x)\equiv 1\}$, we have the following
inequality:

\begin{align}
\int_{Q}[\Delta ^{2},\mu ]z(h\nabla \psi )\leq &~C(\mu
)\int_{0}^{T}\int_{M}|\Delta z|^{2}+l.o.t.^{z}  \notag \\
\leq &~C(\mu )\int_{0}^{T}\int_{\lambda \equiv 1}|\Delta z|^{2}+C(\mu
,T)l.o.t.^{z}  \notag \\
\leq &~C(\mu )\int_{Q}|\Delta \phi |^{2}+C(\mu ,T)l.o.t.^{z}  \tag{3.14}
\end{align}

\subsection{Energy Recovery Estimate}

We may now appeal to our calculation with the $\phi$ multiplier previously,
to obtain our preliminary $\psi$ estimate:

\begin{lemma}[Preliminary $\psi $ Estimate]
\label{psiest} Let $\psi \equiv \mu z$, as
defined above, where $z$ solves (\ref{diff}) with any boundary conditions
under considerations. Moreover, assume $\text{supp}~(\mu )$ is bounded away
from $\Gamma $. Then, in the case of clamped (C) or hinged (H) boundary
conditions we have
\begin{equation*}
\int_{Q}\left( |\psi _{t}|^{2}+|\Delta \psi |^{2}\right) \leq C(\mu ,\lambda
)\Big\{(E_{z}(T)+E_{z}(0))+\int_{Q}\mu \big\{\mathcal{G}(z)+\mathcal{F}(z)%
\big\}(h\nabla \psi )
\end{equation*}
\begin{equation*}
+\int_{Q}\lambda \big\{\mathcal{G}(z)+\mathcal{F}(z)\big\}\phi
+C(T)l.o.t.^{z}  \tag{3.15}  \label{3.15}
\end{equation*}
In the case of free boundary conditions (F)
\begin{equation*}
\int_{0}^{T}\left( ||\psi _{t}||^{2}+a(\psi ,\psi )\right) \leq ~C(\mu
,\lambda )\Big\{(E_{z}(T)+E_{z}(0))+\int_{Q}\mu \big\{\mathcal{G}(z)+%
\mathcal{F}(z)\big\}(h\nabla \psi )
\end{equation*}%
\begin{equation}
+\int_{Q}\lambda \big\{\mathcal{G}(z)+\mathcal{F}(z)\big\}\phi
+C(T)l.o.t.^{z}\Big\}  \tag{3.16}
\end{equation}
\end{lemma}

We note that in the nonlinear boundary term associated with the operator $%
\mathcal{B}_{2}$ vanishes due to the fact that the support of $\mu $ is away
from the boundary.

We may now combine the estimates from Lemma \ref{phiest} and Lemma \ref%
{psiest} to obtain an estimate on the total energy (with either form of
boundary conditions (C) or (H) or (F)):
\begin{equation*}
\int_{0}^{T}\big\{||\psi _{t}||^{2}+||\phi _{t}||^{2}+ a(\phi, \phi ) + a
(\psi, \psi ) +\beta \int_{\Gamma_1} |\phi|^4 \big\}\leq ~C(\mu ,\lambda )%
\Big\{(E_{z}(T)+E_{z}(0))+\int_{Q}\mu \big\{\mathcal{G}(z)+\mathcal{F}(z)%
\big\}(h\nabla \psi )
\end{equation*}%
\begin{equation*}
+\int_{Q}\lambda \big\{\mathcal{G}(z)+\mathcal{F}(z)\big\}%
\phi+\int_{0}^{T}||\phi _{t}||^{2}+C(T,R)l.o.t.^{z}\Big\}  \tag{3.17}
\end{equation*}

\noindent By our choice of supports for $\mu $ and $\lambda $ we note that
the LHS of the above equation \textit{overestimates} the total energy $%
E_{z}(t)$. On the RHS of the estimate we have the term $\int_{Q}|\phi
_{t}|^{2}$, which we replace by $\int_{0}^{T}\int_{\omega }|z_{t}|^{2}$
since $\text{supp}~\lambda \subset \omega $ and on $\text{supp}~\lambda$, $%
\lambda \leq 1$, so we have that
\begin{equation*}
\int_{Q}|\phi _{t}|^{2}\leq \int_{0}^{T}\int_{\omega }|z_{t}|^{2}
\end{equation*}%
Making the appropriate changes above in Lemma \ref{phiest} and Lemma \ref%
{psiest}, we have the analogous result for the free boundary conditions (F).
Hence we can conclude

\begin{lemma}[Preliminary Energy Estimate]
\label{l:rec} For any boundary condition (C), (H), or (F) we have
\begin{equation*}
\int_{0}^{T}E_{z}(t)\leq ~C(\mu ,\lambda )\Big\{(E_{z}(T)+E_{z}(0))+\int_{Q}%
\mu \big\{\mathcal{G}(z)+\mathcal{F}(z)\big\}(h\nabla \psi )+\int_{Q}\lambda %
\big\{\mathcal{G}(z)+\mathcal{F}(z)\big\}\phi
\end{equation*}%
\begin{equation}
+\int_{0}^{T}\int_{\omega }|z_{t}|^{2}+C(T)l.o.t.^{z}  \tag{3.18}
\label{eq:318}
\end{equation}
\end{lemma}

\begin{remark}
At this point, we impose clamped (C) or hinged (H) boundary conditions, in
order to simplify and streamline the analysis. At the end of this section,
we discuss the boundary conditions (F).
\end{remark}

\noindent If we take into account the supports of $\lambda $ and $\mu $
(dropping dependence of the constants on $\mu ,\lambda ,$ and $\Omega $)
then (\ref{eq:318}) with clamped boundary conditions becomes
\begin{equation*}
\int_{0}^{T}E_{z}(t)\leq ~C\Big\{E_{z}(T)+E_{z}(0)+\int_{Q}\big\{\mathcal{G}%
(z)+\mathcal{F}(z)\big\}(h\nabla \psi )+\int_{Q}\big\{\mathcal{G}(z)+%
\mathcal{F}(z)\big\}z \Big\}
\end{equation*}
\begin{equation}
+ C\int_{0}^{T}\int_{\omega }|z_{t}|^{2}+C(T)l.o.t.^{z}  \tag{3.19}
\label{eq:319}
\end{equation}

\begin{remark}
At this point we pause to point out that the estimate we have shown above in
(\ref{eq:319}) will be used in the sections to follow, specifically in the
quasistability estimate. In particular, we must handle the damping terms
(involving $u_t,~w_t$) differently in the estimation for asymptotic
smoothness, versus the estimation for quasistability.
\end{remark}

\noindent By the assumptions on $g$ in Assumption \ref{g}, for every $\delta
$ there exists $C_{\delta }>0$ such that
\begin{equation*}
|u_{t}-w_{t}|^{2}\leq \delta +C_{\delta }\big(g(u_{t})-g(w_{t})\big)\left(
u_{t}-w_{t}\right) .
\end{equation*}%
This gives that
\begin{equation*}
\int_{0}^{T}\int_{\omega }|z_{t}|^{2}\leq T\delta |\Omega |+C(\delta
)\int_{0}^{T}\int_{\omega }(g(u_{t})-g(w_{t}))z_{t}
\end{equation*}%
or, simplifying, and taking into account $\omega \subset \text{supp}~d$
and that $d(\mathbf{x})\geq \alpha _{0}>0$, we have
\begin{equation*}
\int_{0}^{T}\int_{\omega }|z_{t}|^{2}\leq \delta +C(\delta ,T,\Omega
)\int_{Q}\mathcal{G}(z)z_{t}
\end{equation*}%
So taking into account the last inequality in (\ref{eq:319}) , we obtain
\begin{align*}
\int_{0}^{T}E_{z}(t)\leq & ~\delta +C\Big\{E_{z}(T)+E_{z}(0)+\int_{0}^{T}%
\int_{\omega }\big(\mathcal{G}(z)+\mathcal{F}(z)\big)z \\
& ~+\int_{Q}\big(\mathcal{G}(z)+\mathcal{F}(z))h\nabla \psi +C(\delta
,T)\int_{Q}\mathcal{G}(z)z_{t}\Big\} \\
& ~+C(T)l.o.t.^{z}
\end{align*}%
where the constant $C$ does not depend on $T$. Recall, $u$ and $w$ are
solutions to (\ref{plate}) corresponding to some initial conditions $y_{1}$
and $y_{2},$ satisfying $S_{t}y_{1}=(u(t),u_{t}(t))$ and $%
S_{t}y_{2}=(w(t),w_{t}(t))$ for the evolution $S_{t}$ associated to the
plate dynamics. We can assume that $y_{1},~y_{2}\in \mathcal{W}_{R}$ for
some $R>R_{\ast }$, where the invariant set $\mathcal{W}_{R}=\{(u,v)\in
\mathcal{H},$ $\mathscr{E}(u,v)\leq R\}$ . Assuming the solutions $u$ and $w$
are strong, by the invariance of $\mathcal{W}_{R}$ we have
\begin{equation}
||u(t)||_{2}+||u_{t}(t)||+||w(t)||_{2}+||w_{t}(t)||\leq C(R),\text{ \ }t\geq
0  \tag{3.20}
\end{equation}%
\begin{equation}
||u(t)||_{C(\Omega )}+||w(t)||_{C(\Omega )}\leq C(R),\text{ \ }t\geq 0
\tag{3.21}
\end{equation}

Recent developments in the area of Hardy-Lizorkin spaces and compensated
compactness methods allow one to show the following `sharp' regularity of
the Airy stress function $v$:

\begin{theorem}[Sharp Regularity of the Airy Stress Function]
\cite{springer} \label{t:3.4}
\begin{equation*}
||v(u)|| _{W^{2,\infty }}\leq C|| u|| _{2}^{2},\text{ \ \ \ \ \ \ \ }||
v(u,w)|| _{W^{2,\infty }}\leq C|| u|| _{2}|| w|| _{2}
\end{equation*}
where we have denoted $v(u,w)\equiv -\Delta ^{-2}[ u,w]$ and $\mathscr{D}(\Delta^{2} )
= H^4(\Omega) \cap H_0^2(\Omega)$
\end{theorem}

\noindent Making use of the above inequalities, we have the estimate
\begin{equation*}
||[v(u),z]||\leq C||u(t)||_{2}^{2}||z||_{2}\leq C(R)||z||.
\end{equation*}%
Additionally, we have
\begin{equation*}
||v(u)-v(w)||_{W^{2,\infty }}=||v(z,u+w)||_{W^{2,\infty }}\leq
C||z||_{2}\left( ||u||_{2}+||w||_{2}\right) .
\end{equation*}%
Therefore,
\begin{equation*}
||\mathcal{F}%
(z)||=||[u,v(u)]-[w,v(w)]+[z,F_{0}]||=||[v(u)-v(w),z]+[v(w),z]+[z,F_{0}]||%
\leq C(R)||z||_{2},\text{ \ \ \ \ }t\geq 0.
\end{equation*}%
So we obtain
\begin{equation}
\int_{0}^{T}\int_{\omega }\mathcal{F}(z)z\leq \int_{Q}\mathcal{F}(z)z\leq
\epsilon \int_{0}^{T}||z(t)||_{2}^{2}dt+C(T,\epsilon )l.o.t.^{z}  \tag{3.22}
\end{equation}%
and similarly
\begin{equation}
\int_{Q}\mathcal{F}(z)h\nabla \psi \leq C(R)\int_{0}^{T}||z(t)||_{2}||\psi
(t)||_{1}\leq \epsilon \int_{0}^{T}||z(t)||_{2}^{2}+C(T,\epsilon )l.o.t.^{z},
\tag{3.23}
\end{equation}%
(where again, dependence of constants on $\Omega ,\omega ,$ and $h$ are
supressed). To proceed, we need estimates on the dissipation. By the energy
equality
\begin{equation}
E_{z}(T)+\int_{s}^{T}\int_{\Omega }\mathcal{G}(z)z_{t}=E_{z}(s)+\int_{s}^{T}%
\int_{\Omega }\mathcal{F}(z)z_{t},  \tag{3.24}
\end{equation}%
we have
\begin{equation}
\int_{Q}\mathcal{G}(z)z_{t}\leq C(R)+\Big |\int_{Q}\mathcal{F}(z)z_{t} \Big|
\tag{3.25}
\end{equation}%
Taking into account the embedding $H^{2-\eta }(\Omega )\subset C(\Omega )$
for $0<\eta <1,$ we see
\begin{align*}
\int_{Q}\mathcal{G}(z)z\leq & ~\int_{Q}d(\mathbf{x})\big(%
|g(u_{t})|+|g(w_{t})|\big)|z| \\
\leq &~C||z||_{C(0,T;C(\Omega ))}\int_{Q}d(\mathbf{x})\big(%
|g(u_{t})|+|g(w_{t})|\big) \\
\le & ~ C||z||_{C(0,T;H^{2-\eta }(\Omega ))}\int_{Q}d(\mathbf{x})\left(
|g(u_{t})|+|g(w_{t})|\right) .
\end{align*}%
Splitting the region of integration according to $|u_{t}|\leq 1 $ and $%
|u_{t}|> 1,$ and similarly according to $|w_{t}|\leq 1 $ and $|w_{t}|> 1 $,
we obtain
\begin{equation*}
\int_{Q} d(\mathbf{x}) \big( |g(u_t)| + (g(w_t) | \big) \leq g(1)
||d||_{L_{\infty}(\Omega)} \text{meas}( Q)+\int_Q d(\mathbf{x})\big( g(u_t)
u_t + g(w_t) w_t \big) \leq C(R, T)
\end{equation*}
Hence
\begin{equation}
\int_{Q}\mathcal{G}(z)z\leq C(R,T)l.o.t._{1}^{z}  \tag{3.26}
\end{equation}%
Now applying Holder's inequality with the exponent $r>1$ we see
\begin{equation*}
\int_{Q}\mathcal{G}(z)h\nabla \psi \leq C\sup_{[0,T]}||\nabla \psi
(t)||_{r^{\prime }}\int_{Q}d(\mathbf{x})^{r}\left(
|g(u_{t})|^{r}+|g(w_{t})|^{r}\right)
\end{equation*}%
where $\displaystyle \frac{1}{r}+\frac{1}{r^{^{\prime }}}=1$. \ Taking $%
\displaystyle r=1+\frac{1}{p+1}$, and again splitting the region of
integration according to $|u_{t}|\leq 1$ and $|u_{t}|>1$, and using the
polynomial growth condition imposed on $g$ in Assumption 1, we obtain
\begin{equation*}
\int_{Q}d(\mathbf{x})^{r}|g(u_{t})|^{r} \leq C(d)\big\{g(1) \text{meas}(Q)+
\int_Q d(\mathbf{x}) g(u_t) u_t \big\} \leq C(R)(T+1)
\end{equation*}%
Since the same computations hold for terms in $w$, and we have the
continuous embedding $\displaystyle H^{1-\delta }(\Omega )\hookrightarrow
L_{r^{\prime }}(\Omega )$ for sufficiently small $\delta $, we have
\begin{equation}
\int_{Q}\mathcal{G}(z)h\nabla \psi \leq C(R,T)l.o.t._{1}^{z}  \tag{3.27}
\end{equation}%
Hence by the above estimates, we have
\begin{equation*}
\int_{0}^{T}E_{z}(t)\leq C\{E_{z}(T)+E_{z}(0)+\delta +C(R,\delta )+C(\delta
)\int_{Q}\mathcal{F}(z)z_{t}+C(R,T)\big(l.o.t.^{z}+l.o.t._{1}^{z}\big)\}
\end{equation*}%
and eventually by (3.24) we have
\begin{equation}
\int_{0}^{T}E_{z}(t)\leq C_{\ast }\Big\{E(T)+\delta +C(R,\delta )+C(\delta )%
\Big|\int_{0}^{T}\int_{\Omega }\mathcal{F}(z)z_{t}\Big|+C(R,T)\big(%
l.o.t.^{z}+l.o.t._{1}^{z}\big)\Big\}  \tag{3.28}
\end{equation}%
where we write $C_{\ast }$ to emphasize that this constant \textit{does not}
depend on $T$. If we integrate (3.24) over $(0,T)$ with respect to the
variable $s$, and take into account (3.28), we may choose $T$ sufficiently
large ($T>2C_{\ast }$) and $\epsilon $ sufficiently small (with respect to $%
T $) such that

\begin{lemma}[Asymptotic Smoothness Estimate]
\begin{equation}
E_{z}(T)\leq \epsilon +\frac{C(R,\epsilon )}{T}\left( 1+\Big|\int_{Q}%
\mathcal{F}(z)z_{t}\Big|+\Big|\int_{0}^{T}\int_{s}^{T}\int_{\Omega }\mathcal{%
F}(z)z_{t}\Big|\right) +C(\epsilon, R, T)(l.o.t.^{z}+l.o.t._{1}^{z})
\tag{3.29}
\end{equation}
\end{lemma}

We are now in a position to prove Theorem \ref{t:1} on the existence of a
compact attracting set $\mathbf{A}$. For this we shall invoke the abstract
Theorem \ref{psi}. \newline
\newline
\emph{Completion of the proof of Theorem \ref{t:1}}\newline
\newline
To apply Theorem \ref{psi} we need to construct a functional $\Phi
_{\epsilon ,R,T}$ such that
\begin{equation*}
\underset{m\rightarrow \infty }{\lim \inf }\underset{n\rightarrow \infty }{
\lim \inf }\Phi _{\epsilon ,R,T}(y_{n},y_{m})=0
\end{equation*}
for every sequence $\{y_{n}\}$ from B (following from Theorem \ref{psi}).
The functional will contain ``noncompact and not-small" terms in the
inequality (3.29). More specifically, for any initial data $U_0=(u_0,u_1),
~W_0=(w_0,w_1) \in B $ we define
\begin{equation*}
\widetilde{\Phi} _{\epsilon ,R,T}(U_0, W_0) = \Big|\int_0^T (\mathcal{F}(z),
z_t )\Big| +\Big |\int_0^T \int_t^T (\mathcal{F}(z) ,z_t ) \Big|
\end{equation*}
where the trajectory $z = u -w $ has initial data $U_0-W_0$. The key to
compensated compactness is the following representation for the bracket:
\begin{equation}
(\mathcal{F}(z) ,z_t)= \frac{1}{4} \frac{d}{dt} \big\{ - ||\Delta v(u) ||^2
- ||\Delta v(w) ||^2 + 2 ([z,z],F_0) \big\}-\big( [ v(w),w], u_t\big) - %
\big( [ v(u) , u ], w_t\big)  \tag{3.30}
\end{equation}
Integrating the above expression in time and evaluating on the difference of
two solutions $z^{n,m} = w^n - w^m$, where $w^i \rightharpoonup w$ yields:
\begin{align}
\lim_{n\to \infty}\lim_{m\to \infty} \int_t^T (\mathcal{F}(z^{n,m})
,z_t^{n,m} ) =& \dfrac{1}{2} \big\{ ||\Delta v(w)(t) ||^2 - ||\Delta v(w)
(T)||^2 \big\}  \tag{3.31} \\
- \lim_{n \to \infty}& \lim_{m \to \infty} \int_0^T \big\{ \big( [ v(w^n),
w^n ], w_t^m\big) + \big( [ v(w^m), w^m ], w_t^n\big)\big\} ,  \notag
\end{align}
where we have used (a) the weak convergence in $H^2(\Omega)$ of $z^{n,m } $
to 0, and (b) compactness of $\Delta v(w) $ from $H^2(\Omega) \rightarrow
L_2(\Omega) $. The iterated limit in (3.31) is handled via iterated weak
convergence, as follows:
\begin{equation*}
\lim_{n \to \infty} \lim_{m \to \infty} \int_0^T \big\{ \big( [ v(w^n), w^n
], w_t^m\big) + \big( [ v(w^m), w^m ], w_t^n\big)\big\}
\end{equation*}
\begin{equation*}
= 2 \int_t^T ( [ v(w), w] , w_t) = \frac{1}{2} ||\Delta v(w)(t) ||^2 - \frac{%
1}{2} || \Delta v(w)(T) ||^2.
\end{equation*}

\noindent This yields the desired conclusion, that
\begin{equation*}
\lim_{n\rightarrow \infty }\lim_{m\rightarrow \infty }\int_{t}^{T}(\mathcal{F%
}(z^{n,m}),z_{t}^{n,m})=0.
\end{equation*}%
The second integral term in $\widetilde{\Phi }$ is handled similarly. As a
consequence we obtain
\begin{equation*}
\underset{m\rightarrow \infty }{\lim \inf }\underset{n\rightarrow \infty }{%
\lim \inf }\widetilde{\Phi }_{\epsilon ,R,T}(y_{n},y_{m})=0.
\end{equation*}%
\noindent Now, we define
\begin{equation*}
\Phi _{\epsilon ,R,T}=\widetilde{\Phi }+(l.o.t.^{z}+l.o.t._{1}^{z}),
\end{equation*}%
and noting that the terms $(l.o.t.^{z}+l.o.t._{1}^{z})$ in (3.29) are
compact with respect to $H^{2}(\Omega )$ via the Sobolev embeddings, the
final conclusion follows by taking $T$ sufficiently large and $\epsilon $
sufficiently small. This concludes the proof of smoothness estimate required
by Theorem \ref{psi}. Thus, the dynamical system is asymptotically smooth.
In addition, stationary solutions are bounded (due to the inequality in
Lemma 1.1) and the set $\{(u_{0},u_{1})\in \mathcal{H},~{\mathscr{E}}%
(u_{0},u_{1})\leq R\}$ is positively invariant. This implies (\cite{springer}) the existence of local attractors $\mathbf{A}_{\mathbf{R}}$ - the first
statement in Theorem \ref{t:1}. For the second statement we just note that
the \textit{UC} property renders the dynamical system gradient with a
strict Lyapunov function. Thus, asymptotic smoothness and boundedness of the
set of stationary solutions implies the existence of a compact global attractor.

\section{Regularity and finite dimensionality of the Attractor}

Let $\mathbf{A}$ be the global attractor corresponding to the flow $S_{t}$,
as established in Section 3. To prove finiteness of the fractal dimension of
the set $\mathbf{A}$, we shall use Theorem \ref{t:FD} which is based on a
\textit{\ quasistability} estimate. It will turn out that the quasistability
estimate will follow from a more direct estimate written for the difference
of two solutions $z=u-w$ at sufficiently \textit{negative times}. Since the
system is gradient, the latter due to the \textit{UC} property assumed, we have
that trajectories from the attractor stabilize asmyptotically to equilibria
points for both positive and negative times. This is to say: any $x \in
\mathbf{A} $ belongs to some full trajectory $\gamma = \{ (u(t), u_t(t) ) ,
t \in R\} $ and for any $\gamma \subset \mathbf{A }$ there exists a pair $\{
e, e^*\} \subset \mathcal{N} $ (set of equilibria) such that
\begin{equation}  \label{equil}
(u(t), u_t(t) ) \rightarrow (e, 0 ) ~as~ t\rightarrow \infty , ~~(u(t),
u_t(t) ) \rightarrow (e^*, 0 ) ~as~ t\rightarrow - \infty
\end{equation}

\subsection{Quasistability Estimate}

We shall follow a general program developed in \cite{ch-l-jde07} and
supported by PDE estimates derived in previous sections and specific to
localized dissipation.

With the previous notation, we state the following lemma which gives a
preliminary estimate for quasistability inequality:

\begin{lemma}
\label{l:Q} Let $z\equiv u-w$ where $(u(t),u_{t}(t)),~(w(t),w_{t}(t)) \in
\mathbf{A}$ with $w(t) = u(t+h) $, $0 < h\leq1$. Then, there exists $T_u \in
\mathbb{R}$ such that for all $-\infty < s \leq s+T_{0} \leq T_u $ the
following inequality holds:
\begin{equation}
E_{z}(s+T_{0})+\int_{s}^{s+T_{0}}E_{z}(\tau)\leq C(\mathbf{A})T_{0}D_{s}^{s+
T_0}+C(\mathbf{A})T_{0} \sup_{\tau\in \lbrack s,s+T_0]}||z(\tau)||_{2-\eta
}^{2}  \tag{4.2}
\end{equation}
for $\eta >0$, where
\begin{equation*}
D_{t_1}^{t_2}\equiv \int_{t_1}^{t_2}\int_{\Omega }d(\mathbf{x}%
)(g(u_{t})-g(w_{t}))z_{t}
\end{equation*}
\end{lemma}

This lemma will provide a tool for proving additional smoothness of the
attractor. In fact, we will select a time $T$ which will be negative (in
line with the convergence in (\ref{equil})). The above estimate, a
posteriori, will yield the quasistability estimate in Theorem \ref{t:FD},
and hence finite dimensionality of the attractor. More specifically, we
shall follow the program described below after completing the proof of the
above lemma: (1) Apply Lemma \ref{l:Q} with some negative time $s+T $, to
conclude (by homogenity) that the solution at $T$ has additional derivatives
- and is in fact a strong solution. For this step, we will exploit
closedness of velocities to a point of equilibria, along with a
decomposition of the von Karman bracket, which allows us to take the
advantage of this closedness. (2) The above gain of regularity is then
propagated forward on the strength of regularity theory for von Karman
evolutions. This gives an algebraic embedding for the attractor, but
boundedness of the attractor in the higher norm will remain to be
established; (3) This boundedness, that is, the topological embedding, will
follow from a covering theorem, and the compactness of the attractor (shown
in the earlier section). These details are given in \cite{springer}. Thus,
at the end of this procedure we can extend validity of the inequality in
Lemma \ref{l:Q} to all times $-\infty < s < t < \infty$, and we obtain the
conclusion of the additional regularity of the attractor. This then allows
us to infer the \textit{quasistability } estimate in Theorem \ref{t:FD}
(which is the same estimate as in Lemma \ref{l:Q}, but valid for all times $%
s < t $ and all \textit{pairs of trajectories} $(u(t), u_t(t) ) $ and $%
(w(t), w_t(t) ) $ lying in the attractor.

Thus the crux of the proof of regularity and dimensionality of the attractor
reduces to the demonstration of Lemma \ref{l:Q}. We also note that in
comparison with the asymptotic smoothness inequality, the inequality in
Lemma \ref{l:Q} is more demanding. This is due to necessity of keeping
\textit{at least quadratic forms} in the lower order terms. This very demand
forces the damping to have at least linear growth at the origin $g^{\prime
}(0)>0$. (Such restriction is typical - if not necessary - whenever
regularity or finite dimensionality of attractors becomes a concern.)

\subsection{Preparation for the Proof of Lemma \protect\ref{l:Q}}

In the proof of Lemma \ref{l:Q}, we will make use of the recovery estimate
in (\ref{eq:319}) and the energy relation (3.24) as our main tools.
Beginning with (\ref{eq:319}), and taking into account estimates involving $%
\mathcal{F}(z)$ in (3.22), (3.23) and the linear growth condition $g^{\prime
}(0)>0$ in (3.19), we arrive at
\begin{equation*}
\int_0^T \int_{\Omega} |z_t|^2 \leq C D_0^T(z)
\end{equation*}
Applying the above inequality in (\ref{eq:319}) gives:
\begin{equation*}
\int_{0}^{T}E_{z}(\tau ) \leq C\Big\{D_{0}^{T}(z)+E_{z}(T)+E_{z}(0)+\left%
\vert \int_{Q}\mathcal{G}(z)z\right\vert  \notag
\end{equation*}
\begin{equation}  \label{eq:43}
+\left\vert \int_{Q}\mathcal{G}(z)h\nabla z\right\vert \Big\}%
+C(R,T)l.o.t.^{z}  \tag{4.3}
\end{equation}
Now, in tackling quadratic dependence of the dissipation terms, we give the
following proposition

\begin{proposition}
\label{p:1} Let assumptions of Theorem 1.4 be satisfied, and take $z$ be a
solution to (3.2). Then there exists $\delta >0$ such that
\begin{equation}
\left\vert \int_{Q}\mathcal{G}(z)z\right\vert \leq \delta
D_{0}^{T}(z)+C(\delta ,R,T)\underset{[0,T]}{\sup }||z||_{2-\eta
}^{2},~~0<\eta <2-\gamma  \tag{4.4}  \label{eq:44}
\end{equation}%
\begin{equation}
\left\vert \int_{Q}\mathcal{G}(z)h\nabla z\right\vert \leq \delta
D_{0}^{T}(z)+C(\delta ,R,T)\underset{[0,T]}{\sup }||z||_{2-\eta
}^{2},~~0<\eta <1-\gamma  \tag{4.5}  \label{eq:45}
\end{equation}%
where $\mathcal{G}(z)=d(\mathbf{x})\left( g(u_{t})-g(w_{t})\right) $ and $%
D_{0}^{T}(z)=$ $\displaystyle\int_{Q}\mathcal{G}(z)z_{t}.$
\end{proposition}

\begin{proof}
We note that the assumptions on the damping function $g$ (namely, montonicity and the polynomial growth condition in Assumption 1) imply that
\begin{equation}
\label{gs}\tag{4.6}
 \frac{g(s_2) - g(s_1) }{s_2 - s_1} \leq C [ 1 + g(s_1) s_1 + g(s_2) s_2 ]^{\gamma}.
 \end{equation}
 Using the Jensen inequality
 we estimate  \begin{equation*}\left\vert z\right\vert \leq \delta \left\vert
z_{t}\right\vert +C(\delta )\frac{\left\vert z\right\vert
^{2}}{\left\vert z_{t}\right\vert },\end{equation*} The above,
along with (\ref{gs}), gives
\begin{equation*}
\left\vert \int_{Q}\mathcal{G}(z)z\right\vert\leq \delta
D_{0}^{T}(z)+C(\delta ,M)\int_{Q}d(\xb)\left( 1+\left(
g_{0}(u_{t})u_{t}\right) ^{\gamma }+\left(
g_{0}(w_{t})w_{t}\right) ^{\gamma }\right) \left\vert z\right\vert
^{2}
\end{equation*}%
Now, applying the Holder inequality with exponent $p =
\gamma^{-1} $  and Sobolev's embedding $ \ds H^{ 2-\eta} (\Omega)
\subset L_{\frac{2}{1-\gamma}} (\Omega) $, and
taking into account energy equality (\ref{Eident}) we arrive at%
\begin{equation*}
\left\vert \int_{Q}\mathcal{G}(z)z\right\vert\leq \delta
D_{0}^{T}(z)+C(\delta ,R)l.o.t.^z
\end{equation*}%
The inequality in (\ref{eq:45}) can be shown analogously.
\end{proof}
So, taking into account (\ref{eq:44}) and (\ref{eq:45}) in (4.3) we obtain%
\begin{equation*}
\int_{0}^{T}E_{z}(\tau )\leq C\Big\{D_{0}^{T}(z)+E_{z}(T)+E_{z}(0)\Big\}%
+C(R,T)l.o.t.^{z}
\end{equation*}%
We note that $C$ does not depent on $T$, and $l.o.t.^{z}$ is of quadratic
order. By using semigroup property and reiterating the same argument on the
intervals $[s,s+T]$ one obtains
\begin{equation}
\int_{s}^{T+s}E_{z}(\tau )\leq C\Big\{D_{s}^{T+s}(z)+E_{z}(T+s)+E_{z}(s)%
\Big\}+C(R,T)l.o.t.^{z}(s,T+s)  \tag{4.7}  \label{4.7}
\end{equation}%
where $l.o.t.^{z}(s,T+s)$ denote lower order terms collected over the
interval $[s,T+s]$.

In order to prove (4.2), we have to handle the non-compact term $\left(
\mathcal{F}(z),z_{t}\right) $. A technical calculation based on the symmetry
properties of von Karman bracket gives us the following proposition whose
proof is given in \cite{springer}.

\begin{proposition}
If $u,w\in C([0,t];H^{2}(\Omega ))\cap C^{1}([0,t];L_{2}(\Omega ))$ and $%
z=u-w$ then%
\begin{equation}
\left(\mathcal{F}(z),z_{t}\right) =\dfrac{1}{4}\frac{d}{dt}Q(z)+\frac{1}{2}%
P(z)  \tag{4.8}
\end{equation}%
where
\begin{equation*}
Q(z)=\left( v(u)+v(w),[z,z]\right) -||\Delta v(u+w,z)||^2
\end{equation*}%
\begin{equation}
P(z)=-\left( u_{t},[u,v(z,z)]\right) -\left( w_{t},[w,v(z,z)]\right) -\left(
u_{t}+w_{t},[z,v(u+w,z)]\right) .  \tag{4.9}
\end{equation}
\end{proposition}

Now, we can state the following lemma:

\begin{lemma}
\ Let $u(\tau )$ and $w(\tau )$ be two functions from the class
\begin{equation*}
C\big([s,t];H_{0}^{2}(\Omega )\big)\cap C^{1}\big([s,t];L_{2}(\Omega )\big)
\end{equation*}%
for $s,t\in \mathbb{R},$ $s<t,$ such that
\begin{equation*}
||u(\tau )||_{2}^{2}+||u_{t}(\tau )||^{2}\leq R^{2},\text{ \ \ \ \ \ \ }%
||w(\tau )||_{2}^{2}+||w_{t}(\tau )||^{2}\leq R^{2},\text{ \ \ }\tau \in
\lbrack s,t]
\end{equation*}%
Let $z(\tau )=u(\tau )-w(\tau )$ . Then for $\eta >0$
\begin{equation}
\Big|\int_{s}^{t}\left( \mathcal{F}(z),z_{t}\right) \Big|\leq C(R)\underset{%
\tau \in \lbrack s,t]}{\sup }||z(\tau )||_{2-\eta }^{2}+C(R)\int_{s}^{t}\big(%
||u_{t}||+||w_{t}||\big)||z||_{2}^{2}.  \tag{4.10}
\end{equation}
\end{lemma}

\begin{proof}
 The inequality follows from the basic properties of von Karman bracket
 \cite{springer} and the decomposition in Proposition 4.3.
\end{proof}

\subsection{Completion of the Proof of Lemma \protect\ref{l:Q}.}

Let $\gamma =\{(u(t),u_{t}(t)):t\in \mathbb{R\}}$ be a trajectory from the
attractor $\mathbf{A}$, and let $0<h<1.$ It is clear that for the pair $%
w(t)\equiv u(t+h)$ and $u(t)$ satisfy the hypotheses of Lemma 4.4 for every
interval $[s,t]$. We shall estimate the energy $E_{z}(t)$ of $\ z(t)\equiv
z^{h}(t)=u(t+h)-u(t).$ Here we critically use the estimates for the
noncompact term involving $\mathcal{F}(z).$ By (4.10), we have
\begin{equation}
\Big|\int_{s}^{t}\left( \mathcal{F}(z^{h}),z_{t}^{h}\right) \Big|\leq C(R)%
\underset{\tau \in \lbrack s,t]}{\sup }||z^{h}(\tau )||_{2-\eta
}^{2}+C(R)\int_{s}^{t}\big(||u_{t}(\tau +h)||+||u_{t}(\tau )||\big)%
||z^{h}(\tau )||_{2}^{2}.  \tag{4.11}
\end{equation}%
for all $-\infty <s\leq t<+\infty .$ Since we have the characterization $%
\mathbf{A}$ $=M^{u}(\mathcal{N}),$ where $\mathcal{N}$ is the set of
equilibria, we have
\begin{equation*}
\underset{t\rightarrow -\infty }{\lim }d_{H_{0}^{2}(\Omega )\times
L_{2}(\Omega )}(S_{t}W|\mathcal{N}\ )=0\text{ \ \ for any }W\in
H_{0}^{2}(\Omega )\times L_{2}(\Omega )
\end{equation*}%
Hence, for any $\epsilon >0$, there exists $T_{\gamma }^{\epsilon }$
(independent of $h$ but depending on the trajectory $\gamma )$ such that
\begin{equation*}
||u_{t}(\tau )||+||u_{t}(\tau +h)||\leq \epsilon \left[ C(R)\right] ^{-1}%
\text{ \ \ for any }t\leq T_{\gamma }^{\epsilon }.
\end{equation*}%
Taking into account the last inequality in (4.11), we arrive at
\begin{equation}
\Big|\int_{s}^{t}\big(\mathcal{F}(z^{h}),z_{t}^{h}\big)\Big|\leq C(R)%
\underset{\tau \in \lbrack s,t]}{\sup }||z^{h}(\tau )||_{2-\eta
}^{2}+\epsilon \int_{s}^{t}||z^{h}(\tau )||_{2}^{2}.  \tag{4.12}
\end{equation}%
for all $-\infty <s\leq t<T_{\gamma }^{\epsilon }.$ Using the energy
relation (3.24), we find from (4.12) that
\begin{equation}
E_{z}(s)\leq ~E_{z}(t)+\int_{s}^{t}\int_{\Omega }\mathcal{G}(z)z_{t}+C(R)%
\underset{\tau \in \lbrack s,t]}{\sup }||z^{h}(\tau )||_{2-\eta
}^{2}+\epsilon \int_{s}^{t}||z^{h}(\tau )||_{2}^{2}d\tau .  \tag{4.13}
\end{equation}%
and similarly
\begin{equation}
E_{z}(t)\leq ~E_{z}(s)+C(R)\underset{\tau \in \lbrack s,t]}{\sup }%
||z^{h}(\tau )||_{2-\eta }^{2}+\epsilon \int_{s}^{t}||z^{h}(\tau )||_{2}^{2}.
\tag{4.14}
\end{equation}%
for all $-\infty <s\leq t<T_{\gamma }^{\epsilon }.$ Now, if we apply (4.7)
on each subinterval $[s,s+T_{0}]$, we have%
\begin{equation*}
\int_{s}^{s+T_{0}}E_{z}(\tau )\leq C\Big\{D_{s}^{s+T_{0}}(z)+\big( %
E_{z}(s+T_{0})+E_{z}(s)\big) \Big\}+C(R,T_{0})\underset{\tau \in \lbrack
s,s+T_{0}]}{\sup }||z(\tau )||_{2-\eta }^{2}
\end{equation*}%
Taking into account (4.13) in the last inequality and choosing $\epsilon $
sufficiently small we arrive at%
\begin{equation*}
\int_{s}^{s+T_{0}}E_{z}(\tau )\leq C\Big\{%
D_{s}^{s+T_{0}}(z)+E_{z}(s+T_{0})+C(R,T_{0})\sup_{\tau \in \lbrack
s,s+T_{0}]}||z(\tau )||_{2-\eta }^{2}\Big\}
\end{equation*}%
Now, integrating (4.14) and considering the previous inequality we have
\begin{equation*}
E_{z}(s+T_{0})+\int_{s}^{s+T_{0}}E_{z}(\tau )\leq C(\mathbf{A},
T_0)D_{s}^{s+T_{0}}(z)+C(\mathbf{A},T_{0})\underset{\tau \in \lbrack
s,s+T_{0}]}{\sup }||z(\tau )||_{2-\eta }^{2}
\end{equation*}%
which gives (4.2) and thus proves Lemma \ref{l:Q}.

\section{Proof of Theorem \protect\ref{t:2}}

\subsection{Proof of Part (i) in Theorem \protect\ref{t:2}}

Having established Lemma \ref{l:Q} we now proceed with the proof of improved
regularity for the attractor. This is done, as in \cite{ch-l-jde07}, in
three steps:

\textbf{Step 1: Smoothness for negative times}\newline
By (4.2) and energy relation (3.24) written on the interval $[s,s+T_{0}]$ we
can choose a constant $0<\mu <1$ such that $u^{h}(t)=h^{-1}z^{h}(t)$
satisfies the following estimate
\begin{equation}
E_{u^{h}}(s+T_{0})\leq \mu E_{u^{h}}(s)+C(T_{0})\underset{\tau \in \lbrack
0,T_{0}]}{\sup }||u^{h}(s+\tau )||_{2-\eta }^{2}  \tag{5.1}
\end{equation}%
for all $s\leq T_{\gamma }-T_{0},$ where $T_{\gamma }=T_{\gamma }^{\epsilon
_{0}}$ (depending on the trajectory, but not $h$) for some $\epsilon _{0}>0$
and $T_{0}>0.$ Now using interpolation, and taking the supremum over the
interval $(-\infty,T_{\gamma }-T_{0})$ we obtain
\begin{equation*}
\underset{\tau \in (-\infty ,T_{\gamma }]}{\sup }E_{u^{h}}(\tau )\leq \frac{%
1+\mu }{2}\underset{\tau \in (-\infty ,T_{\gamma }]}{\sup }E_{u^{h}}(\tau
)+C(T_{0})
\end{equation*}%
This implies that
\begin{equation}
E_{u^{h}}(s)\leq C(T_{0})\text{ \ \ \ for all }s\in (-\infty ,T_{\gamma }]
\tag{5.2}
\end{equation}%
After passing to the limit $h\rightarrow 0$ we get
\begin{equation}
||u_{tt}(t)||^{2}+||u_{t}(t)||_{2}^{2}\leq C\text{ \ for all }s\in (-\infty
,T_{\gamma }]  \tag{5.3}
\end{equation}%
By (5.3), for $u_{t}\in H^{2}(\Omega )\subset C(\Omega )$, we have $%
g(u_{t})\in C(\Omega )\subset L_{2}(\Omega )$ by the continuity of $g.$
Hence, elliptic regularity theory for $\Delta ^{2}u=-u_{tt}-d(x)g(u_{t})$
with the boundary conditions gives that $\ ||u(t)||_{4}^{2}\leq C$ for all $%
t\in \lbrack -\infty ,T_{\gamma }].$ \vskip.3cm \noindent\textbf{Step 2: Forward
propagation of the regularity }

Using the forward well-posedness of strong solutions stated in Theorem \ref%
{wellp} (and the discussion that follows), we observe that $u(t)$ is a
strong solution to the original problem, and so the global attractor $%
\mathbf{A}$\ is a subset of $\left( H^{4}\cap H_{0}^{2}\right) (\Omega
)\times H_{0}^{2}(\Omega ).$ \vskip.3cm
\noindent\textbf{Step 3: Boundedness of the
attractor in $\left( H^{4}\cap H_{0}^{2}\right)(\Omega)\times
H_{0}^{2}(\Omega ) $ }

In the previous step we have shown that $\mathbf{A}\subset $ $\left(
H^{4}\cap H_{0}^{2}\right) (\Omega )\times H_{0}^{2}(\Omega )$. However,
this does not guarantee the boundedness of $\mathbf{A}$ in $\left( H^{4}\cap
H_{0}^{2}\right) (\Omega )\times H_{0}^{2}(\Omega ).$ So, using the
compactness of the attractor, we will follow an additional argument.

Since for every $\tau \in \mathbb{R}$, the element $u_{t}(\tau )$ belongs to
a compact set in $L_{2}(\Omega )$ that consists of elements from $%
H_{0}^{2}(\Omega ),$ for every $\epsilon >0$ there exists a finite set $%
\{\phi _{j}\}\subset H_{0}^{2}(\Omega )$ such that we can find indices $%
j_{1} $ and $j_{2}$ such that
\begin{equation*}
||u_{t}(\tau )-\phi _{j_{1}}||+||u_{t}(\tau +h)-\phi _{j_{2}}||\leq \epsilon
.
\end{equation*}

Let $P(z)$ be given by (4.9) with the pair $w(t)=u(t+h)$ and $u(t)$ and
\begin{equation*}
P_{j_{1},j_{2}}(z)=-\left( \phi _{j_{1}},[u,v(z,z)]\right) -\left( \phi
_{j_{2}},[w,v(z,z)]\right) -\left( \phi _{j_{1}}+\phi
_{j_{2}},[z,v(u+w,z)]\right)
\end{equation*}%
where $z(t)=u(t+h)-u(t)\equiv z^{h}(t).$ It can be easily shown that
\begin{equation}
||P(z)-P_{j_{1},j_{2}}(z)||\leq \epsilon C(R)||z^{h}(\tau )||_{2}^{2}
\tag{5.4}
\end{equation}

and
\begin{equation*}
||P_{j_{1},j_{2}}(z)||\leq C(R)\left( ||\phi _{j_{1}}||_{2}+||\phi
_{j_{2}}||_{2}\right) ||z^{h}(\tau )||_{2-\eta }^{2}
\end{equation*}%
for $\eta >0.$ So we take
\begin{equation}
\underset{j_{1},j_{2}}{\sup }||P_{j_{1},j_{2}}(z)||\leq
C(\epsilon)||z^{h}(\tau )||_{2-\eta }^{2}  \tag{5.5}
\end{equation}%
for $\eta >0.$ Taking into account (5.4) and (5.5) in (4.8) we see
\begin{equation*}
\underset{t\in \lbrack 0,T]}{\sup }\Big|\int_{s+t}^{s+T}\left( \mathcal{F}%
(z^{h}),z_{t}^{h}\right) \Big|\leq C(\epsilon ,T,R)\underset{\tau \in
\lbrack 0,T]}{\sup }||z^{h}(\tau +s)||_{2-\eta }^{2}+\epsilon
\int_{s}^{t}||z^{h}(\tau )||_{2}^{2}
\end{equation*}%
for all $s\in \mathbb{R}$ with $\eta >0$ and $T>0$. Now applying the above
argument used to prove (5.1) and (5.2), we obtain the boundedness of the
attractor $\mathbf{A}$ in $\left( H^{4}\cap H_{0}^{2}\right) (\Omega )\times
H_{0}^{2}(\Omega ).$ \vskip.3cm This proves the first part of Theorem \ref%
{t:2} .

\subsection{ Proof of (ii) in Theorem \protect\ref{t:2} - finite
dimensionality}

Using the compactness of the attractor and (5.4)-(5.5), we can write
\begin{equation*}
\Big|\int_{s}^{T+s}\left( \mathcal{F}(z),z_{t}\right) \Big |\leq C(\epsilon )%
\big(1+T\big)\underset{\tau \in \lbrack s,s+T]}{\sup }||z^{h}(\tau
)||_{2-\eta }^{2}+\epsilon \int_{s}^{s+T}E_{z}(\tau )
\end{equation*}%
for $s \in R $ and $\eta >0,$ where $z(t)=u(t)-w(t)$ with $(u,u_{t})$ and $%
(w,w_{t})$ from the attractor. Again, applying the same procedure from
above, we obtain (5.1) for $u^{h}=z$ and $s\in \mathbb{R}.$ We then note
that (5.1) yields
\begin{equation*}
E_{z}\big((m+1)T\big)\leq \gamma E_{z}(mT)+C(\mathbf{A},T)b_{m},\text{ \ \ \
}m=0,1,2,...
\end{equation*}%
with $0<\gamma =\gamma (\mathbf{A})<1,$ where
\begin{equation*}
b_{m}\equiv \underset{\tau \in \lbrack mT,(m+1)T]}{\sup }||z(\tau )||^{2}
\end{equation*}%
This yields
\begin{equation*}
E_{z}(mT)\leq \gamma E_{z}(0)+c\sum_{l=1}^{m}\gamma ^{m-l}b_{l-1}
\end{equation*}

\noindent Since $\gamma <1$, there exist constants $C_{1},C_{2}$ and $\sigma
$ possibly depending on $R$ such that for all $t\geq 0$ we have
\begin{equation*}
E_{z}(t)\leq C_{1}E_{z}(0)e^{-\sigma t}+C_{2}\underset{\tau \in \lbrack 0,t]}%
{\sup }||z(\tau )||_{2-\eta }^{2}
\end{equation*}
which yields (2.2). Finally, on the strength of Theorem \ref{t:FD}, applied
with $B =\mathbf{A}$, $\mathcal{H} = D({\mathcal{A}}^{1/2}) \times
L_2(\Omega) $, $\mathcal{H}_1 = H^{2-\eta}(\Omega) \times \{0\} $ we
conclude that $\mathbf{A} $ has a finite fractal dimension.

\end{document}